\documentclass [11pt]{article}
\usepackage{amsmath}
\usepackage{amssymb}
\usepackage{amscd}
\usepackage{amsthm}
\usepackage{amsopn}
\usepackage{xspace}
\usepackage{verbatim}
\usepackage{amsmath}
\usepackage{amsfonts}

\newtheorem{theorem}{Theorem}
\newtheorem{lemma}{Lemma}[section]
\newtheorem{proposition}{Proposition}[section]
\newtheorem{definition}{Definition}[section]
\newtheorem{corollary}{Corollary}[section]

\newtheorem{example}{Example}[section]
\newtheorem{acknowledgment*}{Acknowledgment}
\newtheorem{remark}{Remark}[section]
\numberwithin{equation}{section}

\newcommand{\ccC}{{\cal C}}
\newcommand{\ooM}{{\cal M}_1^+}
\newcommand{\oMj}{{\cal M}^{+,N}}
\newcommand{\D}{{\cal D}}

\newcommand{\oH}{\overline{H}}
\newcommand{\oM}{{\cal M}^+}

\newcommand{\cD}{{\cal D}}
\newcommand{\mm}{{\cal M}_0}

\newcommand{\be}{\begin{equation}}
\newcommand{\ee}{\end{equation}}
\newcommand{\bd}{\begin{displaymath}}
\newcommand{\ed}{\end{displaymath}}

\newcommand{\uE}{\underline{E}}

\newcommand{\cC}{{\cal C}}

\newcommand{\eps}{\varepsilon}
\newcommand{\hC}{\widehat{\cal C}}

\newcommand{\R}{\mathbb R}

\renewcommand{\vec}[1]{\boldsymbol{#1}}

\begin{document}
 \Large  \begin{center}{\bf Limit Theorems for Optimal Mass Transportation}\end{center} \normalsize
\begin{center} G. Wolansky\footnote{This research was supported by M. \& M. Bank Mathematics Res. Fund and by the  Israel Science Foundation }\\ Department of Mathematics, \\
Technion, Haifa 32000, Israel \end{center}
\begin{abstract}
The optimal mass transportation was introduced by  Monge some 200 years ago and is, today, the source of large number of
results in analysis, geometry and convexity.  Here I investigate a new, surprising link between optimal transformations obtained by different Lagrangian actions on Riemannian manifolds.  As a special case, for any pair of non-negative  measures $\lambda^+,\lambda^-$  of equal mass
$$ W_1(\lambda^-, \lambda^+)= \lim_{\eps\rightarrow 0} \eps^{-1}\inf_{\mu} W_p(\mu+\eps\lambda^-, \mu+\eps\lambda^+)$$
where $W_p$, $p\geq 1$  is the Wasserstein distance and  the infimum is over the set of probability measures in the ambient space.
\end{abstract}

\section{Introduction}
The Wasserstein metric $W_p$  ($\infty>p\geq 1$) is a useful distance  on the set of positive Borel measures on metric spaces. Given a metric space $(M, D)$ and a pair of positive Borel measures $\lambda^\pm$ on $M$  satisfying $\int_M d\lambda^+=\int_M d\lambda^-$:
 \be\label{wasserstein}W_p(\lambda^+,\lambda^-):= \inf_\pi\left\{ \left[ \int_M\int_M D^p(x,y) d\pi(x,y)\right]^{1/p} ; \ \ \pi\in {\cal P}(\lambda^+, \lambda^-)\right\} \ , \ee
where ${\cal P}(\lambda^+, \lambda^-)$ stands for the set of all positive Borel measures  on $M\times M$ whose $M-$marginals are  $\lambda^+, \lambda^-$.
\par
Under fairly general conditions (e.g if $M$ is compact),
   a minimizer $\pi^0\in {\cal P}(\lambda^+,\lambda^-)$ of (\ref{wasserstein}) exists. Such minimizers are called {\it  optimal plans}.   I'll  assume in this paper that $M$ is a compact Riemannian manifold and $D$ is a metric related (but not necessarily identical) to the geodesic distance.
\par
If in addition $\lambda^+$ satisfies certain regularity conditions,  the optimal measure $\pi^0$ is supported on a graph of a Borel mapping $\Psi:M\rightarrow M$.  By some abuse of notation we call a Borel map  $\Psi$  {\it an optimal plan} if it is a minimizer of
$$ W_p(\lambda^+,\lambda^-)=\inf_\Phi\left\{  \left[\int D^p(x, \Phi(x)) d\lambda^+\right]^{1/p} \ \ ; \ \ \Phi_\#\lambda^+=\lambda^-\right\}  \  $$
(see Section~\ref{notation}-\ref{diez} for notation).
\par
The metric $W_p$, $p\geq 1$ is a  metrization of the weak topology $C^*(M)$ on positive Borel measures. In particular, it  is continuous in the weak topology.  Thus, it is possible to approximate $W_p(\lambda^+, \lambda^-)$ (and the corresponding optimal plan) by  $W_p(\lambda^+_N,\lambda^-_N)$ on the set of {\it atomic measures}
 \be\label{m+1at}\lambda^\pm_N\in\oMj:=\left\{  \mu =\sum_{i=1}^N m_i\delta_{(x_i)} \ \ \ , m_i\geq 0, \ \ \ x_i\in M \ \right\} \  \  , \  N\rightarrow\infty  \ee reducing (\ref{wasserstein}) into a  {\it finite-dimensional linear programming} on the set of non-negative $N\times N$ matrices $\{{\cal P}_{i,j}\}$ subjected to linear constraints.
\par
There is, however, a sharp distinction between the case $p>1$ and $p=1$. If $p>1$ then the optimal plan $\pi^0$ is unique (for regular $\lambda^+$). This is, in general, {\it not the case} for $p=1$.
Another distinctive feature of the case $p=1$ is its "pinning property": The distance $W_1$ depends only on the difference $\lambda:=\lambda^+-\lambda^-$. This  is manifested by the alternative, dual formulation of $W_1$:
\be\label{w1infi} W_1(\lambda)= \sup_\phi\left\{ \int\phi d\lambda \ \ ; \ \|\phi\|_{Lip}\leq1\right\}\ee
where $\|\phi\|_{Lip}:= \sup_{x\not= y\in M}\left(\phi(x)-\phi(y)\right)/D(x,y)$.
\par\noindent
The optimal potential $\phi$ yields some partial information on the optimal plan $\Psi$ (if exists). In particular, $\nabla \phi(x)$, whenever exists, only indicates {\it the direction} of the optimal plan.
For example, if the metric $D$ is Euclidean, then
 $\Psi(x)=x+t(x)\nabla\phi(x)$ for some  unknown $t(x)\in \R^+$. This is in contrast to the case $p>1$ where a dual variational formulation, analogous to (\ref{w1infi}),    yields the {\it complete information} on the optimal plan $\Psi$ in terms of the gradient of some potential $\phi$.
\par
In this paper I consider an object called  the  $p-$Wasserstein distance $(p>1)$ of $\lambda^+$ to $\lambda^-$, {\it conditioned on a probability measure $\mu$}:
 \be\label{w1phi}W^{(p)}(\lambda\|\mu):= \sup_\phi\left\{ \int \phi d\lambda \ \ ; \ \ \int |\nabla\phi|^q d\mu\leq 1 \right\}\ee
where $q=p/(p-1)$.
\par
The first result  is
\be\label{w1mu} W_1(\lambda)=\min_\mu\left\{ W^{(p)}(\lambda\|\mu) \ \ ; \ \ \int d\mu=1 \right\} \ \ \ , \ (p>1) \ \ee
The problem associated with (\ref{w1mu})  is  related to {\it shape optimization}, see \cite{bu1}. In addition, the minimizer $\mu$ in (\ref{w1mu}) and the corresponding maximizer $\phi$ in (\ref{w1phi}) or (\ref{w1infi}) play
an important rule in the $L_1$ theory of transport  \cite{wil}. In fact, the optimal $\phi$ is, in general, a Lipschitz function which is differentiable $\mu$ a.e.  and satisfies $|\nabla\phi|=1 $  $\mu$ a.e. The minimal measure $\mu$ is called a {\it transport measure}. It verifies the weak form of the continuity equation which, under the current notation,  takes the form
$$ \nabla\cdot (\mu\nabla\phi)= \frac{\lambda}{ W_1(\lambda)} \ . $$
The transport measure yields an additional information on the optimal plan $\Psi$  along the {\it transport rays}  which completes  the information included in $\nabla\phi$ \cite{wil}.  In the context of shape optimization it is related to the optimal distribution of conducting material  \cite{bu1}. See also  \cite{I}, \cite{S1}, \cite{S2}.
\par
The evaluation of the transport measure $\mu$ is therefore an important object of study. It is tempting to approximate the transport measure as a minimizer of (\ref{w1mu}) on a restricted  finite space, e.g. for $\mu\in \oMj$ as defined in (\ref{m+1at}).
\par
 However, this cannot be done. Unlike $W_p$,  $W^{(p)}(\lambda\|\mu)$ is {\it not} continuous in the weak topology of $C^*$ on Borel measures with respect to both $\mu$ and $\lambda$.
  Indeed, it follows easily that $W^{(p)}(\lambda\|\mu)=\infty$ for any atomic measure $\mu$. \par
The second result  of this paper is
\be\label{Ww1pmu} W^{(p)}(\lambda\|\mu)= \lim_{n\rightarrow \infty} nW_p(\mu+\lambda^+/n, \mu+\lambda^-/n)\ee
Here the limit is in the sense of $\Gamma$ convergence. A somewhat stronger result is obtained if we take the infimum over all probability measures $\mu$:
\be\label{Ww1pmumin} W_1(\lambda)= \lim_{n\rightarrow \infty} n \min_{\mu} W_p(\mu+\lambda^+/n, \mu+\lambda^-/n) \ \ee
where the convergence is, this time,  pointwise in $\lambda$.
\par
The importance of (\ref{Ww1pmu}, \ref{Ww1pmumin}) is that $W^{(p)}(\lambda\|\mu)$ can now be approximated by a {\it weakly continuous function}
$$ W^{(p)}_{n}(\lambda^+, \lambda^-\|\mu):= nW_p(\mu+\lambda^+/n, \mu+\lambda^-/n) \ . $$
Suppose $\mu_0$ is a unique minimizer of (\ref{w1mu}).
If $\mu_n$ is a minimizer of $W^{(p)}_{n}(\lambda^+, \lambda^-\|\mu)$  then the sequence $\{\mu_n\}$ must converge to the transport measure $\mu_0$.
In contrast to $W^{(p)}$, $W^{(p)}_n$ {\it is continuous} in the $C^*$ topology with respect to $\mu$. Hence $\mu_n$ can be approximated by  atomic measures  $\mu_n^N\in\oMj$ (\ref{m+1at}).
  In particular a transport measure can be approximated  by  a finite {\it  points allocation}
obtained by minimizing $W^{(p)}_{n}$ on $\oMj$ for a sufficiently large $n$ and  $N$.
\par
The results (\ref{w1mu}- \ref{Ww1pmumin}) can be extended to the case where the cost $D^p$ on  $M\times M$  is generalized into an action function on a Riemannian manifold $M\times M$, induced by a Lagrangian function $l:TM\rightarrow \R$.  This point of view reveals some relations with the {\it Weak KAM Theory}
dealing with invariant measures of Lagrangian flows on manifolds.
\subsection{Overview}
Section~\ref{background}  review the necessary background for the Weak KAM and its relation to optimal transport.  Section~\ref{mainresults}  state the main results (Theorems~\ref{th1}-\ref{th6}), which correspond to   (\ref{w1mu}- \ref{Ww1pmumin}) for homogeneous Lagrangian on
$M\times M$.
Section~\ref{proof1} presents the proof of  the first of the main results which generalizes (\ref{w1phi}). Finally, Section~\ref{th23} contains the proofs of the other main results which generalize (\ref{Ww1pmu}, \ref{Ww1pmumin}).
\subsection{Standing notations and assumptions}\label{notation}
\begin{enumerate}
\item\label{Mg}  $(M,g)$ is a compact, Riemannian Manifold and $D:M\times M\rightarrow \R^+$ is the geodesic distance.
\item \label{Pi} $TM$ (res. $T^*M$) the tangent (res. cotangent) bundle of $M$. The duality between $v\in T_xM$ and $p\in T_x^*M$ is denoted by $\langle \xi, v\rangle\in \R$. The projection $\Pi:TM\rightarrow M$ is the trivialization $\Pi(x,v)=x$. Likewise $\Pi^*:T^*M\rightarrow M$ is the trivialization $\Pi^*(x,\xi)=x$.
\item For any topological space $X$, ${\cal M}(X)$ is the set of Borel measures on $X$, ${\cal M}_0(X)\subset {\cal M}(X)$ the set of such measures which are perpendicular to the constants. ${\cal M}^+(X)\subset{\cal M}(X)$ the set
    of all non-negative measures in ${\cal M}$, and  $\ooM(X)\subset{\cal M}^+(X)$ the set of normalized (probability) measures. If $X=M$, the parameter $X$ is usually omitted.
    \item \label{diez} A Borel map $\Phi:X_1\rightarrow X_2$ induces a mapping $\Phi_\#: {\cal M}^+(X_1)\rightarrow {\cal M}^+(X_2)$ via
        $$ \Phi_\#(\mu_1)(A)= \mu_1(\Phi^{-1}(A))$$
        for any Borel set $A\subset X_2$.
        \item For any \label{kappa} $x,y\in M$ let ${\cal K}^T_{x,y}$ be the set of all absolutely continuous paths $\vec{z}:[0,T]\rightarrow M$ connecting $x$ to $y$, that is, $\vec{z}(0)=x$, $\vec{z}(T)=y$.
\item\label{kappadef} Given $\mu_1, \mu_2\in {\cal M}^+$, the set ${\cal P}(\mu_1,\mu_2)$ is defined as all the measures $\Lambda\in {\cal M}^+(M\times M)$ such that $\pi_{1, \#}\Lambda=\mu_1$ and $\pi_{2, \#}\Lambda=\mu_2$, where $\pi_i:{\cal M}\times{\cal M}\rightarrow{\cal M}$ defined by $\pi_1(x,y)=x$, $\pi_2(x,y)=y$.
\item \label{hamdes} An hamiltonian function $h\in C^2(T^*M; \R)$ is assumed to be strictly convex and super-linear in $\xi$ on the fibers $T^*_xM$, uniformly in  $x\in M$, that is
     $$ h(x,\xi) \geq -C+ \hat{h}(\xi) \ \ \text{where} \lim_{\|\xi\|\rightarrow\infty}\hat{h}(\xi)/\|\xi\|=\infty\ \ \ \ . $$
     The Lagrangian $l:TM\rightarrow \R$ is obtained by Legendre duality
     $$ l(x,v)=\sup_{\xi\in T^*_xM} \langle \xi, v\rangle-h(x,\xi)$$
     satisfies $l\in C^2(TM; \R)$, and is super linear on the fibers of $T_xM$ uniformly in $x$.
     \item\label{Exp} $Exp_{(l)}:TM\times \R\rightarrow M$ is the flow due to the Lagrangian $l$ on $M$, corresponding to the  Euler-Lagrange equation
         $$ \frac{d}{dt} l_v = l_x \ . $$
         For each $t\in\R$, $Exp^{(t)}_{(l)}:TM\rightarrow M$ is the exponential map at time $t$.

\end{enumerate}
\section{Background}\label{background}
 The weak version of Mather's theory \cite{ma}  deals with minimal invariant measures of  Lagrangians, and the corresponding Hamiltonians defined on a manifold $M$.  In this theory the concept of an orbit
$\vec{z}=\vec{z}(t):\R\rightarrow M$ is replaced by that of a {\it closed probability measure} on $TM$:

\be\label{closedm}{\cal M}^c_0:= \left\{ \nu\in {\cal M}^+_1(TM) \ ;  \ \ \int_{TM} l(x,v) d\nu(x,v) < \infty \ , \ \ \int_{TM}\langle d\phi, v\rangle d\nu =0\ \ \text{for  \ any} \ \phi\in C^1(M) \right\}\ . \ee

A minimal (or Mather)  measure $\nu_M\in {\cal M}^c_0$ is a minimizer of
\be\label{minlag} \inf_{\nu\in {\cal M}^c_0}\int_{TM} l(x,v) d\nu(x,v) := -\uE\  \ee
It can be shown (\cite{[Ba]}, \cite{[Fa]}, \cite{[P]}) that
 any minimizer of (\ref{minlag}) is invariant under the flow induced by the Euler-Lagrange equation on $TM$:
\be\label{eleq} \frac{d}{dt} \nabla_{\dot{x}} l(x, \dot{x}) = \nabla_x l(x,\dot{x}) \ . \ee
\par
There is also a dual formulation of (\ref{minlag}) \cite{f2}, \cite{wol1}:
\be\label{maxham} \sup_{\mu\in \ooM} \inf_{\phi\in C^1(M)} \int_M h(x, d\phi) d\mu=\uE \ , \ee
where the maximizer $\mu_M$ is the projection of a Mather measure $\nu_M$ on $M$.
 The ground energy level $\uE$, common to (\ref{minlag}, \ref{maxham}), admits several equivalent definitions.
Evans and Gomes (\cite{ev} \cite{eg1} \cite{eg2}) defined $\uE$ as the {\it effective hamiltonian value}
$$ \uE:= \inf_{\phi\in C^1(M)} \sup_{x\in M} h(x, d\phi) \ , $$
while the PDE approach to the WKAM theory  (\cite{f1}, \cite{f2}) defines $\uE$ as the minimal $E\in \R$ for which the Hamilton-Jacobi equation
$h(x, d\phi)=E $
admits a viscosity sub-solution on $M$. Alternatively $\uE$ is the {\it only} constant for which
$h(x, d\phi)=\uE$
admits a viscosity solution \cite{f0}.
There are other, equivalent definitions of $\uE$ known in the literature. We shall meet some of them below.
\begin{example}
\begin{description}
\item{i)} \ $l=l_K:= |v|^p/(p-1)$ where $p>1$. Here $\uE=0$ and  $\mu_M$ is the volume induced by the metric $g$.
\item{ii)}  $l(x,v)=(1/2)|v|^2 -V(x)$ where $V\in C^2(M)$ (mechanical Lagrangian) .  Then
$\uE=\max_{x\in M} V(x)$ and  $\mu_M$ of (\ref{maxham}) is supported at the points of maxima of $V$.
\item{iii)} \   $l(x,v)=l_K( v-\vec{W}(x))$ where $\vec{W}$ is a section in $TM$.  \\
    Then  (\ref{minlag}) implies $\uE\leq 0$. In fact, it can be shown that $\uE=0$ for {\it any} choice of $\vec{W}$.
    \item{iv)}  In general, if $\vec{P}$ is in the first cohomology of $M$ ($\mathbf{H}^1(M)$) then  $l \mapsto l(x,v)-\langle\vec{P}, v\rangle$ induced the hamiltonian $h\mapsto h(x, \xi+\vec{P})$ and  $\uE=\alpha(\vec{P})$ corresponds to the celebrated Mather ($\alpha$) function \cite{ma} on the cohomology $\mathbf{H}^1(M)$. See also \cite{S}.
\end{description}
\end{example}
The Monge problem of mass transportation, on the other hand, has a much longer history. Some years before the
the French revolution, Monge (1781) proposed to consider the minimal cost of transporting a given mass distribution to another, where the cost of transporting a unit of mass from point $x$ to $y$ is prescribed by a function $C(x,y)$.  In modern language, the Monge problem on a manifold $M$ is described as follows: Given a pair of Borel probability measures
$\mu_0, \mu_1$ on $M$, consider the set ${\cal K}(\mu_0,\mu_1)$ of all Borel mappings $\Phi:M\rightarrow M$ transporting $\mu_0$ to $\mu_1$, i.e
$$ \Phi\in {\cal K}(\mu_0, \mu_1) \Longleftrightarrow \Phi_\#\mu_0=\mu_1$$ and look for the one which minimize the {\it transportation cost}
\be\label{mongedef} \cC(\mu_0, \mu_1):= \inf_\Phi\left\{  \int_M C(x, \Phi(x)) d\mu_0(x) \ \ ; \ \ \Phi\in{\cal K}(\mu_0, \mu_1)\right\}  \ . \ee
In this generality, the set ${\cal K}(\mu_0, \mu_1)$ can be empty if, e.g., $\mu_0$ contains an atomic measure while $\mu_1$ does not, so
$C(\mu_0,\mu_1)=\infty$ in that case. In 1942, Kantorovich proposed a relaxation of this deterministic definition of the Monge cost. Instead of the (very nonlinear)  set ${\cal K}(\mu_0, \mu_1)$, he suggested to consider the set ${\cal P}(\mu_0,\mu_1)$ defined in section~\ref{notation}-(\ref{kappadef}).  Then, the definition of the Monge metric is relaxed into the linear optimization
\be\label{kantdef}\cC(\mu_0,\mu_1)= \min_{\Lambda}\left\{ \int_{M\times M} C(x,y) d\Lambda(x,y)  \ \ ; \ \ \Lambda\in {\cal P}(\mu_0, \mu_1)\right\}\ . \ee
\begin{example}
 The {\it Wasserstein} distance $W_p$ ($p\geq 1$) is obtained by the power $p$ of the metric $D$ induced by the Riemannian structure:
 \be\label{wasser}W_p(\mu_0,\mu_1) =   \min_{\Lambda} \left\{ \left[\int_{M\times M} D^p(x,y) d\Lambda(x,y)\right]^{1/p} \ ; \ \Lambda\in {\cal P}(\mu_0, \mu_1)\right\}\ee
 \end{example}
The advantage of this relaxed definition is that $C(\mu_0,\mu_1)$ is always finite, and that a minimizer of
(\ref{kantdef}) always exists by the compactness of the set ${\cal P}(\mu_0,\mu_1)$ in the weak topology $C^*(M\times M)$. If $\mu_0$ contains no atomic points then it can be shown that $C(\mu_0,\mu_1)'s$ given by  (\ref{mongedef}) and (\ref{kantdef}) coincide \cite{ampar}.
\par
The theory of Monge-Kantorovich (M-K) was developed in the last few decades in a countless number of publications. For  updated reference see \cite{wil}, \cite{vi1}. \footnote{ By convention, the name "Monge problem" is reserved for the metric cost, while "Monge-Kantorovich problem" is usually referred to general cost functions}

Returning now to WKAM, it was observed by Bernard and  Buffoni   (\cite{bb1}\cite{bb2}- see also \cite{wol1})  that the  minimal measure and the ground energy
can be expressed in terms of  the  M-K problem subjected to the cost function
induced by the Lagrangian (recall section~\ref{notation}-\ref{kappa})
 \be\label{CTdef}C_T(x,y):= \inf_{\vec{z}}\left\{ \int_0^T l\left(\vec{z}(s); \dot{\vec{z}}(s)\right) ds  \  \ , \
 \vec{z}\in {\cal K}^T_{x,y}\right\} \ , T>0 \ .  \ee
Then
$$ \cC_T(\mu):= \cC_T(\mu,\mu)=\min_{\Lambda} \left\{ \int_{M\times M} C_T(x,y) d\Lambda(x,y) \ \ ; \ \
\Lambda\in {\cal P}(\mu,\mu)\right\}  $$
 and
\be\label{mkmin} \min_{\mu} \left\{ \cC_T(\mu) \ ; \ \mu\in\ooM\right\}= -T\uE\ee
 where the minimizers of (\ref{mkmin}) coincide, for any $T>0$, with the projected Mather measure $\mu_M$ maximizing  (\ref{maxham}) \cite{bb2}.
The action $C_T$ induces a metric on the manifold $M$:
\be\label{dedef1} (x,y)\in M\times M \mapsto D_E(x,y)=\inf_{T>0} C_T(x,y) + TE \ .  \ee
\newpage
\begin{example} \ \begin{description}
\item{i)} For $l(x,v)=|v|^p/(p-1)$, $p>1$  we get  $C_T(x,y)= D(x,y)^p/(p-1)T^{p-1}$ while $D_E(x,y)= pE^{1-1/p}D_g(x,y)/(p-1)$ if $E\geq 0$, $D_E(x,y)=-\infty$ if $E<0$.
    \item{ii)}  $l(x,v)=(1/2)|v|^2 -V(x)$ where $V\in C^2(M)$ (mechanical Lagrangian) .  Then
    $D_E(x,y)$ is the geodesic distance induced by conformal equivalent metric $(M, (E-V)g)$ on $M$, where $E\geq\uE= \sup_M V$.
    \end{description}
\end{example}
It is not difficult to see that either $D_E(x,x)=0$ for any $x\in M$, or $D_E(x,y)=-\infty$ for any $x,y\in M$.
In fact, it follows (\cite{mane1}, \cite{cdi}) that $D_E(x,y)=-\infty$ for $E<\uE$ and $D_E(x,x)=0$ for $E\geq \uE$ and any $x,y\in M$.

 Let now  $\lambda^+\ , \lambda^-\in {\cal M}^+$   where  $\lambda:=\lambda^+-\lambda^-\in {\cal M}_0$, that is $\int_M d\lambda=0$. Let
   \be\label{mk}\D_E(\lambda):=\D_E(\lambda^+,\lambda^-)= \min_{\Lambda} \left\{ \int_{M\times M} D_E(x,y) d\Lambda(x,y)  \  \ ; \ \Lambda\in{\cal P}(\lambda)\right\}  \ee
  be the Monge distance of $\lambda^+$ and $\lambda^-$ with respect to the metric $D_E$. There is a dual formulation of $\D_E$ as follows: Consider the set
  ${\cal L}_E$ of $D_E$ Lipschitz functions on $M$:
  \be\label{LE} {\cal L}_E:= \left\{ \phi\in C(M) \ ; \ \ \phi(x)-\phi(y)\leq D_E(x,y) \ \ \forall \ x,y\in M\right\}\ee
  Then (see, e.g \cite{wil}, \cite{vi})
  \be\label{dualdefD} \D_E(\lambda)= \max_\phi\left\{\int_M \phi d\lambda \  \ ; \ \phi\in {\cal L}_E . \right\} \ee
\section{Description of the main results}\label{mainresults}
The object of this paper is to establish some relations between the action $C_T$ and a modified  action $\hC_T$ defined below.
\subsection{Unconditional action}
For given $\lambda\in\mm$ we generalize (\ref{closedm}) into
\be\label{closedmslambda}{\cal M}_\lambda:= \left\{ \nu\in {\cal M}^+_1(TM) \ ;  \ \ \int_{TM} l(x,v) d\nu(x,v) < \infty \ \ ; \ \ \int_{TM}\langle d\phi, v\rangle d\nu =\int_M \phi d\lambda \ \ \text{for  \ any} \ \phi\in C^1(M) \right\}\  \ee
and define
\be\label{minlag1} \hC(\lambda):= \inf_{\nu}\left\{\int_{TM} l(x,v) d\nu(x,v) \ \ ; \ \ \nu\in {\cal M}_\lambda \right\} \  . \ee

The modified action $\hC_T:\mm\rightarrow \R\cup\{\infty\}$, $T>0$  have several equivalent definitions as given in Theorem~\ref{th1} below:
\begin{theorem}\label{th1}  The following definitions are equivalent:
\begin{enumerate}
\item\label{1}$ \hC_T(\lambda):=T\hC\left(\frac{\lambda}{T}\right) \   . $
\item\label{hcT} $\hC_T(\lambda):= \min_{\mu}\sup_{\phi}\left\{ \int_M-Th(x, d\phi)d\mu + \phi d\lambda \ \ ; \ \ \mu\in\ooM  , \ \phi\in C^1(M) \right\} \ .  $
    \item\label{hcT1} $\hC_T(\lambda):= \max_{E\geq \uE}\left[ \D_E(\lambda)- ET\right] \ . $
\end{enumerate}
In addition if  $T_c:= D_{\uE}^{'+}(\lambda) <\infty$ then for $T\geq T_c$,
 $$ \hC_T(\lambda)= \hC_{T_c}(\lambda)- T\uE \ . $$
 In that case the minimizer
 $\mu_\lambda^T\in\ooM$of (\ref{hcT1}), $T>T_c$ is given by
 $$ \mu_\lambda^T= \frac{T_c}{T} \mu_\lambda^{T_c} + \left( 1-\frac{T_c}{T}\right)\mu_M  \ , $$
 where $\mu_M$ is the projected Mather measure.
\end{theorem}
\begin{remark}
Note that $\D_E(\lambda)$ (\ref{mk}, \ref{dualdefD}) is a monotone non-decreasing  and concave function of $E$ while $\D_{\uE}(\lambda)>-\infty$ by definition. Hence the right-derivative of  $\D^{'+}_E(\lambda)$  as a function of $E$ is defined and positive (possibly $+\infty$ at $E=\uE$).
 \par
 \end{remark}
\begin{remark}
A special case of Theorem~\ref{th1} was introduced in \cite{wol2}.
\end{remark}
\par
 For the next result  we need a two technical  assumptions:

\vskip .2in\noindent
${\bf H_1}$ \ \ \ There exists a sequence of smooth, positive mollifiers $\delta_\eps: M\times M\rightarrow \R^+$ such that, for any $\phi\in C^0(M)$ (res. $\phi\in C^1(M)$)
$$ \lim_{\eps\rightarrow 0}\delta_\eps * \phi= \phi$$
where the convergence is in $C^0(M)$ (res. $C^1(M)$) and for any $\eps>0$ and $\phi\in C^1(M)$
$$  \delta_\eps * d\phi=d(\delta_\eps*\phi)  \ . $$
\par\noindent ${\bf H_2}$ \ \ \
For any $(x,p)\in T^*M$ and  $\eps>0$ there exists $\delta>0$ such that $h(x,\xi)-h(y,\xi_y)\leq \eps (h(x,\xi)+1)$  provided $D(x,y)<\delta$. Here $\xi_y$ is obtained  by parallel translation of $(x,\xi)$ to $y$.

 \begin{remark}
${\bf H_1}$ holds for  homogeneous spaces,  e.g the flat $d-$torus $\R^d/\mathbb{Z}^n$ or the sphere $\mathbb{S}^{d-1}=SO(d)/SO(1)$.
\\
${\bf H_2}$ holds, in particular, for any mechanical hamiltonian with continuous potential.
\end{remark}

\vskip .2in
\begin{theorem}\label{th3} Assume ${\bf H_1}+ {\bf H_2}$.
 For any $\lambda=\lambda^+-\lambda^-$ where $\lambda^\pm\in\ooM$, $$\hC_T(\lambda)=\lim_{\eps\rightarrow 0}\min_{\mu\in \ooM} \eps^{-1}\cC_{ \eps T}(\mu+\eps\lambda^-, \mu+\eps\lambda^+) \ . $$
\end{theorem}

As an application of Theorem~\ref{th3} we may consider the case where the lagrangian $l$ is homogeneous  with respect to a Riemannian metric $g_{(x)}$:

\vskip .3in\noindent
\begin{example} \label{lh}If $l(x,v)= |v|^p/(p-1)$
where $p>1$.
Then $C_T(x,y) = \frac{D^p(x,y)}{(p-1)T^{p-1}}$ while \\ $D_E(x,y)= \frac{p}{p-1}E^{(p-1)/p}D(x,y)$   and $\uE=0$.
It follows that
\be\label{lhf} \hC_T(\lambda)= \frac{
W_1^p(\lambda)}{(p-1)T^{p-1}} , \ \ \ \eps^{-1}\cC_{ \eps T}(\mu+\eps\lambda^+, \mu+\eps\lambda^-)=\frac{
W_p^p(\mu+\eps\lambda^+, \mu+\eps\lambda^-)}{(p-1)T^{p-1}\eps^{p}} \ee
where the Wasserstein distance $W_p$  is defined in (\ref{wasser}).
Hence, by
Theorem~\ref{th1} and Theorem~\ref{th3}
$$ W_1(\lambda)=\lim_{\eps\rightarrow 0} \eps^{-1}\inf_{\mu\in\ooM} W_p(\mu+\eps\lambda^-, \mu+\eps\lambda^+) \ . $$
\end{example}
 \begin{remark}
 The optimal transport description of the weak KAM theory (\ref{mkmin}) can be considered as a special case of Theorem~\ref{th3} where $\lambda=0$. Indeed $\inf_{\mu\in\ooM} \eps^{-1}C_{\eps T}(\mu, \mu)=-T\uE$ by (\ref{mkmin}). On the other hand, since $\D_E(0)=0$ for any $E\geq\uE$ it follows that $T_c=0$, hence   $\hC_{T_c}(0)=0$ so $\hC_T(0)=-T\uE$ as well  by  the last part of Theorem~\ref{th1}.
 \end{remark}
\subsection{Conditional action}
There is also an interest in the definition of action (and metric distance) conditioned with a given probability measure $\mu\in\ooM$. We  introduce these definitions and reformulate parts of the main results Theorems~\ref{th1}-\ref{th3}
in terms of these.
\par
For a given $\mu\in\ooM$ and $E\geq \uE$, let
\be\label{HEmore} {\cal H}_E(\mu):= \left\{ \phi\in C^1(M) \ ; \ \ \int_M h(x, d\phi) d\mu \leq E \ \right\} \ . \ee
In analogy with (\ref{dualdefD}) we define the $\mu-${\it conditional metric} on  $\lambda\in {\cal M}_0$:
 \be\label{Demudef}\D_E(\lambda\|\mu):= \sup_\phi\left\{ \int_M \phi d\lambda  \ \ ; \ \ \phi\in {\cal H}_E(\mu)\right\}  \ . \ee

The {\em conditioned, modified action} with respect to $\mu\in\ooM$ is defined in analogy with Theorem~\ref{th1} (\ref{hcT}, \ref{hcT1})
\be\label{CTmudef} \hC_T(\lambda\|\mu):= \max_{E\geq \uE} \D_E(\lambda\|\mu)- ET\equiv \sup_{\phi\in C^1(M)} \int_M-Th(x, d\phi) d\mu + \phi d\lambda \ . \ee

\par\noindent
 \begin{example}\label{lh11} As in Example~\ref{lh},  $l(x,v)= |v|^p/(p-1)$ implies $h(\xi)=q^{-q}|\xi|^q$ where $q=p/(p-1)$. Then (\ref{HEmore}, \ref{Demudef})  is related  to (\ref{w1phi}), that is $W_1^{(p)}(\lambda\|\mu)=\cD_E(\lambda\|\mu)$
 where $E=q^{-q}$ or
  \be\label{lh11f}\D_E(\lambda\|\mu)=qE^{1/q}W_1^{(p)}(\lambda\|\mu) \ , \ \ \hC_T(\lambda\|\mu)=\frac{q-1}{T^{1/(q-1)}} \left( W_1^{(p)}(\lambda\|\mu)\right)^p\ee
 \end{example}

\begin{remark}
It seems there is a relation between this definition and the tangential gradient \cite{bu}. There are also possible applications to {\it optimal network} and {irrigation} theory, where one wishes to minimize $D(\lambda\|\mu)$ over some constrained set of $\mu\in\ooM$  (the irrigation network) for a prescribed $\lambda$ (representing the set of sources and targets).
See, e.g. \cite{bu2}, \cite{bcm} and the ref. within.
\end{remark}
The next result is
 \begin{theorem}\label{th5}
 For any $\lambda\in \mm$,
 $$ \D_E(\lambda)=\min_{\mu\in\ooM} \D_E(\lambda\|\mu) \ , \ \ \ \hC_T(\lambda)=\min_{\mu\in\ooM} \hC_T(\lambda\|\mu) \ .  $$
 \end{theorem}

\noindent
The analog of Theorem~\ref{th3} holds for the conditional action as well. However, we can only prove the $\Gamma-$convergence in that case. Recall that a sequence of functionals $F_n:{\bf X}_n\rightarrow \R\cup\{\infty\}$   is said to $\Gamma-$converge to $F: {\bf X}\rightarrow \R\cup\{\infty\}$ ($\Gamma-\lim_{n\rightarrow \infty} F_n=F$) if and only if
\begin{description}
\item{(i)} ${\bf X}_n\subset {\bf X}$ for any $n$.
\item{(ii)} For any sequence $x_n\in {\bf X}_n$ converging to $x\in {\bf X}$ in the topology of ${\bf X}$ ,
$$ \liminf_{n\rightarrow\infty} F_n(x_n) \geq F(x) \ . $$
\item{(iii)} For any $x\in {\bf X}$ there exists a sequence $\hat{x}_n\in {\bf X}_n$ converging to $x\in {\bf X}$ in the topology of ${\bf X}$ for which
    $$ \lim_{n\rightarrow\infty} F_n(\hat{x}_n) = F(x) \ . $$
\end{description}

In Theorem~\ref{th6} below the $\Gamma-$convergence is related to the special case where ${\bf X}_n= {\bf X}$:
 \begin{theorem}\label{th6}
 Let ${\bf X}_n=\mm\times \ooM= {\bf X}$ and $F_n(\lambda,\mu):= n\cC_{T/n}(\mu+\lambda^-/n, \mu+\lambda^+/n)$. Then
$$\hC_T(\cdot\|\cdot)=  \Gamma- \lim_{n\rightarrow \infty}F_n \  . $$
\end{theorem}
From Theorem~\ref{th6} and Theorem~\ref{th3} it follows immediately
\begin{corollary}\label{coradd}
In addition, if $\mu_n$ is a minimizer of $F_n$ in $\ooM$ then any converging subsequence of $\mu_n$, $n\rightarrow \infty$, converges to a minimizer of $\hC(\lambda\|\cdot)$ in $\ooM$.
\end{corollary}

  Finally, we note that (\ref{Ww1pmumin}) is a special case of Theorem~\ref{th6}.
 Using  Examples~\ref{lh}, \ref{lh11} with $\eps=1/n$, recalling $(q-1)^{-1}=p-1$ we obtain
 \begin{corollary}\label{corspecial}
 $$ W_1(\lambda)= \lim_{n\rightarrow\infty} n \min_{\mu\in\ooM} W_p(\mu+\lambda^+/n, \mu+\lambda^-/n)$$
 \end{corollary}

\section{Proof of Theorems~\ref{th1}\&\ref{th5}}\label{proof1}
We first show that $\hC(\lambda)<\infty$ (recall  (\ref{minlag1})).
 \begin{lemma}\label{5.1} For any $\lambda\in{\cal M}_0$, \
 ${\cal M}_\lambda\not=\emptyset$. In particular, since the Lagrangian $l$ is bounded from below,
 $\hC(\lambda)<\infty$.
 \end{lemma}
\begin{proof}
It is enough to show that there exists a compact set $K\subset TM$ and  a sequence $\{ \lambda_n\}\subset {\cal M}_0$  converging weakly to $\lambda$ such that for each $n$ there exists $\nu_n\in {\cal M}_{\lambda_n}$ whose support is contained in $K$. Indeed, such a set is compact and there exists a weak limit $\nu=\lim_{n\rightarrow\infty}\nu_n$ which satisfies $\lim_{n\rightarrow \infty}v\nu_n=v\nu$ as well. Hence, if $\phi\in C^1(M)$ then
$$ \lim_{n\rightarrow\infty} \int_M\langle d\phi, v\rangle d\nu_n =\int_M\langle d\phi, v\rangle d\nu \
\ \ \ , \ \ \ \lim_{n\rightarrow\infty} \int_M \phi d\lambda_n= \int_M \phi d\lambda_n \ . $$
Since $\nu_n\in {\cal M}_{\lambda_n}$ we get
$$ \int_M\langle d\phi, v\rangle d\nu_n=\int_M \phi d\lambda_n$$
for any $n$, so the same equality holds for $\nu$ as well.

Now, we consider
\be\label{lambdan}\lambda_n=\alpha_n\sum_{j=1}^n \left(\delta_{x_j}-\delta_{y_j}\right)\ee
where $x_j, y_j\in M$ and $\alpha_n>0$. For any pair $(x_j, y_j)$ consider a geodesic arc corresponding to the Riemannian metric which connect $x$ to $y$, parameterized by the arc length: $\vec{z}_j:[0,1]\rightarrow M$ and $|\dot{\vec{z}}|= D(x_j,y_j)$ (recall section~\ref{notation}-(\ref{Mg})).
 Then
$$ \nu_n:=\alpha_n\sum_{j=1}^n\int_0^1 \delta_{x-\vec{z}_j(t), v-\dot{\vec{z}}_j(t)} dt \  $$
satisfies for any $\phi\in C^1(M)$
 \begin{multline}\int_M \langle d\phi, v\rangle d\nu_n=\alpha_n\sum_{j=1}^n\int_0^1\langle d\phi\left( \vec{z}_j(s), \dot{\vec{z}}_j(s)\right) \dot{\vec{z}}_j(t) \rangle dt = \alpha_n\sum_{j=1}^n\int_0^1\frac{d}{dt}\phi\left( \vec{z}_j(s)\right) dt \\ = \alpha_n\sum_{j=1}^n\left[\phi(y_j)-\phi(x_j)\right] = \int_M\phi d\lambda_n \
 \end{multline}
 hence $\nu_n\in {\cal M}_{\lambda_n}$. Finally, we can certainly find such a sequence $\lambda_n$ of the form (\ref{lambdan}) which converges weakly to $\lambda$.
\end{proof}
\subsection{Point  distances and Hamiltonians}
For $E\in\R$, let $\sigma_E:TM\rightarrow \R$ the support function of the level surface $h(x,\xi)\leq E$, that is:
\be\label{sigmadef} \sigma_E(x,v):= \sup_{\xi\in T^*_xM} \left\{ \langle \xi,v\rangle_{(x)} \ ; \ h(x,\xi)\leq E \right\}\ . \ee
It follows from our standing assumptions (Section~\ref{notation}-\ref{hamdes}) that
 $\sigma_E$ is differentiable as a function of $E$ for any $(x,v)\in TM$.  For the following Lemma see e.g. \cite{Ro}.

 Recall that
\be\label{DEdef} D_E(x,y):=\inf_{T>0} C_T(x,y) + ET\ee where $C_T$ as defined in (\ref{CTdef}). Recall also
 section~\ref{notation}-\ref{kappa}:

\begin{lemma} . \label{basicdef}
\be\label{DEdef2}  D_E(x,y)=\inf_{\vec{z}\in{\cal K}^1_{x,y}} \int_0^1\sigma_E\left(\vec{z}(s), \dot{\vec{z}}(s)\right)ds \ .  \ee
\end{lemma}
Given $x\in M$, let
\be\label{uE2} \uE:= \inf\left\{ E\in\R; \ D_E(x,x)>-\infty \right\}\ee
For the  following Lemma see \cite{mane} (also \cite{S}):
\begin{lemma}
$\uE$ is independent of $x\in M$. The definitions (\ref{uE2}) and (\ref{minlag}) and (\ref{maxham}) are equivalent. If $E\geq \uE$ then $D_E(x,y) > -\infty$ for any $x,y\in M$ and, in addition
\begin{description}
\item{i)} $D_E(x,x)=0$ for any $x\in M$.
\item{ii)} For any $x,y,z\in M$, $D_E(x,z)\leq D_E(x,y) + D_E(y,z)$
\end{description}
\end{lemma}
From (\ref{DEdef}), Lemma~\ref{basicdef} and the continuity of $\sigma_E$ with respect to $E\geq \uE$ we get
\begin{corollary}\label{diff}
If $E\geq \uE$ then for any $x,y\in M$,  $D_E(x,y)$ is continuous, monotone non-decreasing and concave as a function of $E$.
\end{corollary}

Note that the differentiability of $\sigma_E$ with respect to $E$ does {\em not} imply that $D_E(x,y)$ is differentiable for each $x,y\in M$. However, since $D_E(x,y)$ is a concave function of $E$ for each $x,y\in M$, it is differentiable for Lebesgue almost any $E>\uE$. We then obtain by differentiation
\begin{lemma}\label{time}
If $E$ is a point of differentiability of $D_E(x,y)$ then there exists a geodesic arc $\vec{z}\in {\cal K}_{x,y}^1$
realizing (\ref{DEdef2}) such that the $E$ derivative of $D_E(x,y)$ is given by
\be\label{realE} T_E(x,y):= \frac{d}{dE} D_E(x,y)= \int_0^1 \sigma^{'}_E\left(\vec{z}(s), \dot{\vec{z}}(s)\right)ds \ , \ee
where $\sigma^{'}_E$ is the $E$ derivative of $\sigma_E$.  Moreover
\be\label{de=te} D_E(x,y)= C_{T_E(x,y)}(x,y) + ET_E(x,y) \ . \ee
\end{lemma}
From (\ref{sigmadef}) we get $\sigma_E(x, v)\leq |v|\max\{ |p| \ ; \ h(x,\xi)\leq E\}$. From our standing assumption on $h$ (section~\ref{notation}-(\ref{hamdes})) and (\ref{DEdef2}) we obtain
\begin{lemma} For any $x,y\in M$ and $E\geq \uE$
 $$D_E(x,y)\leq  \hat{h}^{-1}(E+C) D(x,y)$$
 In particular
\be\label{Dgestim} \lim_{E\rightarrow\infty} E^{-1}D_E(x,y)=0\ee
uniformly on $M\times M$.
\end{lemma}
\begin{corollary}\label{corlip} For $E\geq \uE$, the set ${\cal L}_E$ (\ref{LE}) is contained in the set of Lipschitz functions with respect to $D$, and ${\cal L}_E$ is locally compact in $C(M)$.
\end{corollary}

Given $\phi\in C^1(M)$ let \be\label{oHdef}
\overline{H}(\phi):= \sup_{x\in M} h(
x, d\phi) \  \ .  \ee
We extend the definition of $\oH$ to the larger class of Lipschitz functions by the following

\begin{lemma}\label{6.3}
 If $\phi\in C^1(M)$ then $$ \overline{H}(\phi)
=
\min_{E \geq \underline{E}}\left \{E; \
\phi\in  {\cal L}_E\right\}
\ ,
$$
where ${\cal L}_E$ as defined in (\ref{LE}).
 \end{lemma}
\begin{proof}
First we show that if $\phi\in {\cal L}_E\cap C^1(M)$ then $h(x, d\phi)\leq E$ for all $x\in M$. Indeed, for any  $x,y\in M$ and any  curve $z(\cdot)$ connecting  $x$ to $y$
$$ \phi(y)-\phi(x) = \int_0^1d \phi(z(t))\cdot \dot{z} dt\leq D_E(x,y)\leq \int_0^1\sigma_E(z(t), \dot{z}(t)) dt$$
hence $d\phi(x)\cdot v\leq \sigma_E(x,v)$ for any $v\in T_xM$. Then, by definition, $d\phi(x)$
is contained in any supporting
half  space which contains the set $Q_x(E):= \{ \xi\in T^*_xM; \ h(x,\xi)\leq E\}$. Since this set is convex by assumption, it follows that $d\phi\in Q_x(E)$, so $h(x, d\phi)\leq E$ for any $x\in M$. Hence $\overline{H}(\phi)\leq E$.

Next we show the opposite inequality $h(x, d\phi)\geq E$ for all $x\in M$.  Recall (\ref{de=te}). Then
 for any $\eps>0$ we can find $T_\eps>0$ and $\vec{z}_\eps\in {\cal K}^{T_\eps}_{x,y}$ so
\be\label{est1} D_E(x,y)\geq \int_0^{T_\eps}l(\vec{z}_\eps(t), \dot{\vec{z}}_\eps(t)) dt + (E-\eps)T_\eps  \ . \ee
 Next, for a.e $t\in[0, T_\eps]$
 \be\label{est2}h\left(\vec{z}_\eps(t), d\phi(\vec{z}_\eps(t))\right) \geq \dot{\vec{z}}_\eps(t)\cdot d\phi(\vec{z}_\eps(t))-l\left(\vec{z}_\eps(t), \cdot{\vec{z}}_\eps(t) \right) \ . \ee
Integrate (\ref{est2}) from $0$ to $T_\eps$ and use $\vec{z}_\eps\in{\cal K}^{T_\eps}_{x,y}$, (\ref{est1}, \ref{est2}) and the definition of ${\cal L}_E$ to obtain
$$ T_\eps^{-1}\int_0^{T_\eps}h\left(\vec{z}_\eps(t), d\phi(\vec{z}_\eps(t))\right)dt \geq T_\eps^{-1}\left[\phi(y)-\phi(x)\right]-T_\eps^{-1}\int_0^{T_\eps}l\left(\vec{z}_\eps(t), \cdot{\vec{z}}_\eps(t) \right)dt
\geq E-\eps \ . $$
Hence, the supremum of $h(x, d\phi)$ along the orbit of $\vec{z}_\eps$ is, at least, $E-\eps$.
Since $\eps$ is arbitrary, then $\oH(\phi)\geq E$.
 \end{proof}
 \subsection{Measure  distances and Hamiltonians}
From Lemma~\ref{6.3} and Corollary~\ref{corlip}  we  extend the definition of $\oH$ to the space $Lip(M)$
of Lipschitz functions on $M$.
Let now define for $\lambda\in\mm$
 \be\label{H*def}\oH_T^*(\lambda):= \sup_{
\phi\in Lip(M)}\left\{ -T\oH(
\phi)  + \int_M \phi d\lambda\right\}
 \in \R\cup\{\infty\} \  . \ee

 \begin{proposition}\label{cormain} For any $\lambda\in \mm$
\be\label{W1e} \oH_T^*( \lambda)
 =
 \sup_{E\geq \underline{E}}\left\{\D_E(\lambda)-TE\right\}  \ . \ee
 \end{proposition}
 \begin{proof}
 By definition of $\overline{H}^*$ and Lemma~\ref{6.3},
 \begin{multline}\label{chain1} \oH_T^*(\lambda)=
 \sup_{\phi\in Lip(M)}\left[
 \int_M \phi d\lambda- T\oH(\phi)
 \right] =
\sup_{\phi\in Lip(M)}\sup_{E\geq \underline{E}}\left[\int_M \phi d\lambda-TE
 \ ; \
 \phi\in {\cal L}_E\right]
 \\
 = \sup_{E\geq \underline{E}}\sup_{\phi\in Lip(M)}\left[ \int_M\phi d\lambda-TE \ ; \ \phi\in {\cal L}_E\right] = \sup_{E\geq \underline{E}}\left\{\D_E(\lambda)-TE\right\}, \end{multline}
 where we used the duality relation given by (\ref{dualdefD}).
\end{proof}

\begin{corollary}\label{corh*}
$\overline{H}_T^*$ is weakly continuous on ${\cal M}_0$.
\end{corollary}
\begin{proof}
For each $E\geq \uE$, the Monge-Kantorovich metric $\D_E:\mm\rightarrow \R$ is continuous on $\mm$ (under weak* topology).
Indeed, it is u.s.c. by  (\ref{mk}) and l.s.c. by the dual formulation (\ref{dualdefD}).

Also, for each $\lambda\in\ooM$, $\D_E(\lambda)$ is concave and finite in $E$  for $E\geq \uE$. It follows  that $\D$ is mutually continuous on $[\uE, \infty[ \times \mm$.
From (\ref{Dgestim}) we also get that $\D$ is coercive on $\mm$, that is $\lim_{E\rightarrow\infty} E^{-1}\D_E(\lambda)=0$
locally uniformly on $\mm$. These imply that
$\oH^*_T$ is continuous on $\mm$ via (\ref{W1e}).
\end{proof}

We return now to Corollary ~\ref{diff} and Lemma~\ref{time}. It follows that  for any countable dense set $A\subset M$ there exists a (possibly empty) set $N\subset ]\uE, \infty[$ of zero Lebesgue measure such that   $D_E(x,y)$ is differentiable in $E\in ]\uE, \infty[ - N$, for {\it any} $x,y\in A$. Let ${\cal M}(A)\subset {\cal M}_0$ be the set of all measures in ${\cal M}_0$ which are supported on a {\it finite} subset of $A$, and such that $\lambda(\{x\})$ is rational for any $x\in A$. Again, since ${\cal M}(A)$ is countable, it follows by Corollary~\ref{diff} that $\D_E(\lambda)$ is differentiable (as a function of $E$) for any $\lambda\in {\cal M}(A)$ and any $E\in ]\uE, \infty[ - N$ for a (perhaps larger) set $N$ of zero Lebesgue measure. It is also evident that ${\cal M}_0$ is the weak closure of ${\cal M}(A)$.

\begin{lemma}\label{mongk}
For any $\lambda^+-\lambda^-\equiv \lambda\in {\cal M}(A)$ and $E\in ]\uE, \infty[ - N$, there exists an optimal plan $\Lambda^\lambda_E\in{\cal P}(\lambda^+,\lambda^-)$ realizing
\be\label{mk1} \int_{M\times M} D_E(x,y) d\Lambda^\lambda_E(x,y) = \min_{\Lambda\in {\cal P}(\lambda^+,\lambda^-)} \int_{M\times M} D_E(x,y) d\Lambda(x,y)\equiv \D_E(\lambda)\ee
for which
\be\label{tineq0} \frac{d}{dE} \D_E(\lambda)=\sum_{x,y\in A}\Lambda^\lambda_E(\{x,y\}) T_E(x,y) \ . \ee
  \end{lemma}
\begin{proof}
Let $E_n\searrow E$. For each $n$, set $\Lambda_{E_n}^\lambda$ be a minimizer of  (\ref{mk1}) subjected to $E=E_n$.   We choose a subsequence so that the
limit
\begin{equation}\label{limae} \Lambda_{E^+}^\lambda(\{x,y\}):= \lim_{n\rightarrow\infty} \Lambda_{E_n}^\lambda(\{x,y\})\end{equation} exists for any $x,y\in A$.
Evidently,  $ \Lambda_{E^+}^\lambda\in {\cal P}(\lambda^+,\lambda^-)$ is an optimal plan for (\ref{mk1}).
Next,

 $$\D_{E_n}(\lambda)-\D_E(\lambda) \geq  \sum_{x,y\in A}\Lambda_{E_n}^\lambda(\{x,y\})  \left(
 D_{E_n}(x,y) - D_{E}(x,y)\right)$$
  Divide by $E_n-E>0$ and let $n\rightarrow\infty$, using (\ref{limae}) and (\ref{realE}) we get
 \be\label{tineq1}\frac{d}{dE} \D_E(\lambda) \geq \sum_{x,y\in A}\Lambda_{E^+}^\lambda(\{x,y\})T_E(x,y) \ . \ee
 We repeat the same argument for a sequence $E^n \nearrow E$ for which $$ \Lambda_{E^-}^\lambda(\{x,y\}):= \lim_{n\rightarrow\infty} \Lambda_{E_n}^\lambda(\{x,y\}) $$
 and get
 \be\label{tineq2} \frac{d}{dE} \D_E(\lambda) \leq \sum_{x,y\in A}\Lambda_{E^-}^\lambda(\{x,y\})T_E(x,y) \ . \ee
 Again $\Lambda_{E^-}^\lambda$ is an optimal plan as well. If $\Lambda_{E^-}^\lambda=\Lambda_{E^+}^\lambda$ then we are done. Otherwise, define $\Lambda_{E^-}^\lambda$ as a convex combination of $\Lambda_{E^-}^\lambda$ and $\Lambda_{E^+}^\lambda$ for which the equality (\ref{tineq0}) holds due to (\ref{tineq1}, \ref{tineq2}).
 \end{proof}
Given $x,y\in M$, let $E$ be a point of differentiability of $D_E(x,y)$, and $\vec{z}^E_{x,y}:[0,1]\rightarrow M$ a geodesic arc connecting $x,y$ and realizing (\ref{realE}). Then $d\tau^E_{x,y}:= \sigma^{'}_E\left(\vec{z}^E_{x,y}, \dot{\vec{z}}^E_{x,y}\right) ds$ is a non-negative measure on $[0,1]$,  and  (\ref{realE}) reads $T_E(x,y)=\int_0^1 d\tau^E_{x,y}$. Let
$\mu^E_{x,y}$ be the measure on $M$ obtained by pushing $\tau^E_{x,y}$ from $[0,1]$ to $M$ via $\vec{z}^E_{x,y}$:
$$\mu^E_{x,y}:= \left(\vec{z}^E_{x,y}\right)_\# \tau^E_{x,y}\in {\cal M}^+ \ , $$
that is, for any $\phi\in C(M)$,
\be\label{muedef} \int_M \phi d \mu^E_{x,y}  := \int_0^1\phi\left(\vec{z}^E_{x,y}(t)\right) d\tau^E_{x,y} \ . \ee
\begin{definition}\label{defmu}
For any $\lambda\in {\cal M}(A)$  and $E\in ]\uE,\infty[ - N$ let
$$ \mu^E_\Lambda:= \sum_{x,y\in A}\Lambda_E^\lambda(\{x,y\})\mu^E_{x,y}$$
where $\mu^E_{x,y}$ are as given in  (\ref{muedef}) and $\Lambda^\lambda_E$ is the particular optimal plan given in Lemma~\ref{mongk}.
\end{definition}
\begin{remark}\label{remnor}Note that $\int_M d\mu^E_\Lambda = \D^{'}_E(\lambda)$ for any $\lambda\in {\cal M}_0(A)$ and $E\in ]\uE, \infty[ - N$ by Lemma~\ref{mongk}, where $\D^{'}_E(\lambda)= (d/dE)\D_E(\lambda)$.
\end{remark}
\begin{definition}
For any  $\lambda\in {\cal M}_0$, $T>0$,
 $E(\lambda,T)$ is the maximizer of (\ref{W1e}), that is
$$ \D_{E(\lambda,T)}(\lambda) - TE(\lambda,T)\equiv  \oH_T^*(\lambda) \ . $$
\end{definition}
By Corollary~\ref{diff} (in particular, the concavity of $\D_E(\lambda)$ with $E$) we obtain
\begin{lemma}\label{2side}
If $E(\lambda,T)>\uE$ then
$$\left. \frac{d^+}{dE}\D_E(\lambda)\right|_{E=E(\lambda,T)}\leq T \leq \left. \frac{d^-}{dE}\D_E(\lambda,T)\right|_{E=E(\lambda,T)}$$
where $d^+/dE$ (res. $d^-/dE$) stands for the right (res. left) derivative. If $E(\lambda,T)=\uE$ then
$$\left. \frac{d^+}{dE}\D_E(\lambda)\right|_{E=\uE}\leq T \ . $$
\end{lemma}

\subsection
{ Proof of Theorem~\ref{th1} \ (\ref{1} $\leftrightarrows$\ref{hcT}) }
First we note that it is enough to assume $T=1$. Consider
\be\label{calF} {\cal F}(\mu,\phi):=  \int_M-h(x, d\phi)d\mu + \phi d\lambda  \  \ee
where $\lambda\in\mm$ is prescribed.
Evidently, ${\cal F}$ is convex lower semi continuous (l.s.c) in $\mu$ on  $\ooM$ and concave upper semi continuous (u.s.c) in $\phi$ on $C^1(M)$. Since $\ooM$ is compact, the Minimax Theorem implies
 \be\label{minmaxF}\sup_{\phi\in C^1(M)}\min_{\mu\in \ooM} {\cal F}(\mu,\phi)= \min_{\mu\in \ooM} \sup_{\phi\in C^1(M)}{\cal F}(\mu,\phi) \ . \ee
Next define
$$ {\cal G}(\nu, \phi):= \int_{TM} \left( l(x,v)- \langle d\phi,  v\rangle \right) d\nu +\int_M \phi d\lambda \ . $$
on  $\ooM(TM)\times C^1(M)$. Then (recall (\ref{closedmslambda}))
 \be\label{minmax2}\sup_{\phi\in C^1(M)} \inf_{\nu\in \ooM(TM)} {\cal G}(\nu,\phi) \leq \inf_{\nu\in{\cal M}_\lambda} \int_{TM} l(x,v) d\nu\equiv \hC(\lambda) \ . \ee
Now
$$ \overline{\cal G}(\nu):= \sup_{\phi\in C^1(M)} {\cal G}(\nu, \phi) \equiv \left\{ \begin{array}{cc}
                                                                                       \int_{TM} l(x,v) d\nu  & \text{if} \ \nu\in {\cal M}_\lambda \\
                                                                                       \infty  &  \text{if} \ \nu\not\in {\cal M}_\lambda
                                                                                     \end{array}\right. \ .  $$
We recall, again, from the Minmax Theorem that the inequality in (\ref{minmax2}) turns into an equality provided
the set $  \{ \nu\in \ooM(TM); \ \ \overline{\cal G}(\nu) \leq \hC(\lambda)\}$ is compact.  However $\hC(\nu)<\infty$ by Lemma~\ref{5.1}. Since $l$ is super linear in $v$ uniformly in $x$ (see section~\ref{notation}-\ref{hamdes})
it follows that the sub-level set $\{\nu\in {\cal M}_\lambda; \ \int_{TM} l(x,v) d\nu\leq C < \infty\}$ is tight for any constant $C$, hence  compact.

Next
 \begin{multline}\label{chain} \int_{TM}\left(l(x,v) -\langle d\phi, v\rangle\right)d\nu(x,v)+ \int_M\phi d\lambda
  \\
= \int_{M}\phi d\lambda -h(x, d\phi)d\mu+ \int_{TM}\left(l(x,v) -\langle d\phi, v\rangle+h(x,d\phi)\right)d\nu(x,v)
 . \end{multline}
 where $\mu = \Pi_\# \nu$.
By the Young inequality  $l(x,v)+h(x,\xi) \geq \langle \xi, v\rangle_{(x)}$  for any $\xi\in T_x^*M$, $v\in T_xM$ with equality if and only if
$v= h_\xi(x,d\phi(x))$. So, the second term on the right of (\ref{chain}) is non-negative, but, for any $\mu\in\ooM$
$$ \inf_{\nu}\left\{ \int_{TM}\left(l(x,v) -\langle d\phi, v\rangle\right)d\nu(x,v)\ \ ; \ \nu\in \ooM(TM) \ , \Pi_\#\nu=\mu\right\}= -\int_{M}h(x, d\phi)d\mu$$
is realized for $\nu=\delta_{v-h_\xi(x,d\phi(x))}\oplus \mu\in \ooM(TM)$.  From this and (\ref{chain}) we obtain
$$ \inf_{\nu\in\ooM(TM)} {\cal G}(\nu, \phi)= \inf_{\mu\in \ooM} {\cal F}(\phi,\mu)$$
hence
$$ \sup_{\phi\in C^1(M)}\inf_{\nu\in\ooM(TM)} {\cal G}(\nu, \phi)= \sup_{\phi\in C^1(M)}\inf_{\mu\in\ooM}{\cal F}(\phi,\mu)= \hC(\lambda)$$
and this part of the Theorem follows from (\ref{minmaxF}).

 $\Box$

\subsection{Proof of Theorem \ref{th1}:(\ref{hcT}$\leftrightarrows$\ref{hcT1}) }
We now define, for {\it any} $\lambda\in {\cal M}_0$, a measure $\mu_\lambda\in \ooM$ in the following way:

Assume, for now, that $\lambda\in {\cal M}(A)$. If $E(\lambda,T)\in ]\uE, \infty[-N$ then define $\mu_\lambda=\mu^{E(\lambda,T)}_\Lambda$ according to Definition~\ref{defmu}. Otherwise, fix a sequence  $E^n\in ]\uE, \infty[ - N$ such that $E^n\searrow E(\lambda,T)$.
Similarly, let $E_n\in ]\uE, \infty[ - N$ such that $E_n\nearrow E(\lambda,T)$.

Then $\mu_{\Lambda_n}^{E^n}$ and $\mu_{\Lambda_n}^{E_n}$ are given by Definition~\ref{defmu} for any $n$. Let  $\mu_\lambda^+$ be a weak limit of the sequence $\mu_{\Lambda_n}^{E^n}$, and, similarly, $\mu_\lambda^-$ be a weak limit of the sequence $\mu_{\Lambda_n}^{E_n}$.

 By Lemma~\ref{2side} and Remark~\ref{remnor} we get
 \be\label{mupm} \int_M d\mu_\lambda^+\leq T \leq \int_M d\mu_\lambda^-  \ . \ee If $E(\lambda,T)=\uE$ then we can still define $\mu_\lambda^+$, and it satisfies the left inequality of (\ref{mupm}).

\begin{definition}\label{mulambda}
For any $\lambda\in {\cal M}_0$, let $\mu_\lambda$ defined in the following way:
\begin{description}
\item {i)} If $\lambda\in {\cal M}_0(A)$ then
\begin{itemize}
\item If $E(\lambda,T)>\uE$ then $\mu_\lambda$ is a convex combination of $T^{-1}\mu_\lambda^+, T^{-1}\mu_\lambda^-$ given by (\ref{mupm}) such that $\mu_\lambda\in \ooM$ (that is, $\int d\mu_\lambda,=1$).
    \item If $E(\lambda,T)=\uE$ then
    \be\label{part3} \mu_\lambda= T^{-1}\mu_\lambda^+ + \left( 1-T^{-1}\int_M d\mu_\lambda^+\right) \mu_M\ee
    where $\mu_M$ is a projected Mather measure.
\end{itemize}
\item{ii)} For $\lambda\not\in {\cal M}_0(A)$, let $\lambda_n\in {\cal M}_0(A)$ be a sequence converging weakly to $\lambda$. Then $\{\mu_\lambda\}$ is the set of weak limits of the sequence $\mu_{\lambda_n}$.
\end{description}
\end{definition}
Define
 \be\label{calEdef}{\cal Q}(\lambda, \mu):= \sup_{
\phi\in C^1(M)}\left\{ -\int_M h(x, d\phi) d\mu
 +\int_M \phi d\lambda \right\}
  \in \R\cup\{\infty\} \  , \ \ {\cal Q}_T(\lambda,\mu):= {\cal Q}(\lambda, T\mu) \ .  \ee
  Recall from \ref{1}$\leftrightarrows$\ref{hcT} that
  \be\label{11} \hC_T(\lambda)=\inf_{\mu\in \ooM}{\cal Q}_T(\lambda, \mu)\equiv \inf_{\mu\in \ooM}{\cal Q}(\lambda, T\mu) \ . \ee
  Also, from (\ref{H*def}), (\ref{oHdef}) and Proposition~\ref{cormain}
  \be\label{22}  \oH_T^*(\lambda) \leq {\cal Q}_T(\lambda,\mu) \ \ \   \forall \mu\in\ooM \ .  \ee
  We have to show that
  \be\label{oh=cale}\oH_T^*(\lambda)=\inf_{\mu\in\ooM} {\cal Q}_T(\lambda,\mu) \  \ee
  for any $\lambda\in \mm$.
  It is enough to prove (\ref{oh=cale}) for a dense set of in ${\cal M}_0$, say for any $\lambda\in {\cal M}_0(A)$.  Suppose (\ref{oh=cale}) holds for a sequence $\{\lambda_n\}\subset {\cal M}_0(A)$ converging weakly to
  $\lambda\in {\cal M}_0$, that is, $\oH_T^*(\lambda_n)=\hC_T(\lambda_n)$. Since $\oH_T^*$ is weakly continuous by Corollary~\ref{corh*} we get
  $\oH_T^*(\lambda)=\lim_{n\rightarrow\infty}\oH_T^*(\lambda_n)$. On the other hand we recall that, according to definition~\ref{hcT} of Theorem~\ref{th1},  $\hC_T: {\cal M}_0 \mapsto \R$ is l.s.c. So $\lim_{n\rightarrow\infty}
  \hC_T(\lambda_n)\geq \hC_T(\lambda)$, hence $\oH_T^*(\lambda)\geq \hC_T(\lambda)$. By (\ref{11}, \ref{22})
  we get (\ref{oh=cale}) for {\it any} $\lambda\in\mm$.
    \par
  The proof of $\ref{hcT}\leftrightarrows\ref{hcT1}$ then follows from
\begin{lemma}\label{lemma3.5}
For any $\lambda\in {\cal M}_0(A)$
\be\label{ntp}{\cal Q}_T(\lambda, \mu_\lambda)=\oH_T^*(\lambda) \  \ee
holds where $\mu_\lambda\in \ooM$ is as given in Definition~\ref{mulambda}.
\end{lemma}
\begin{proof}

Let $\lambda\in{\cal M}_0(A)$ and $E\in ]\uE, \infty[-N$. Then we use (\ref{muedef}) for {\it any} $\phi\in C^1(M)$
$$ -\int_M h(x, d\phi) d\mu_\Lambda^E= -\sum_{x,y\in A} \Lambda(\{x,y\})\int_0^1h\left(\vec{z}^{E}_{x,y}(s),
d\phi\left(\vec{z}^{E}_{x,y}(s)\right)\right) ds $$
We now perform a change of variables $ds\rightarrow dt=\sigma^{'}_E\left(\vec{z}^{E}_{x,y}(s), \dot{\vec{z}}^{E}_{x,y}(s)\right) ds$ which transforms the interval $[0,1]$ into $[0, T_E(x,y)]$ (see (\ref{realE})) and we get
$$ -\int_M h(x, d\phi) d\mu_\Lambda^E= -\sum_{x,y\in A} \Lambda(\{x,y\})\int_0^{T_E(x,y)}h\left(\widehat{\vec{z}}^{E}_{x,y}(t),
d\phi\left(\widehat{\vec{z}}^{E}_{x,y}(t)\right)\right) dt $$
where $\widehat{\vec{z}}^E_{x,y}$ is the re-parametrization of $\vec{z}^E_{x,y}$, satisfying $\widehat{\vec{z}}^E_{x,y}(0)=x$, $\widehat{\vec{z}}^E_{x,y}(T_E(x,y))=y$. Next
$$\int_M\phi d\lambda= \int_M d\Lambda^E_\lambda(x,y)\left[\phi(y)-\phi(x)\right]
= \sum_{x,y\in A} \Lambda(\{x,y\})\int_0^{T_E(x,y)}d\phi\left(\widehat{\vec{z}}^E_{x,y}(t)\right) \dot{\widehat{\vec{z}}}^E_{x,y}(t) dt  $$
so $\int_M\phi d \lambda- \int_M h(x, d\phi) d\mu_\Lambda^E = $
\begin{multline}\label{longineq} \sum_{x,y\in A} \Lambda^E_\lambda(\{x,y\})\int_0^{T_E(x,y)}
\left[ d\phi\left(\widehat{\vec{z}}^E_{x,y}(t)\right) \dot{\widehat{\vec{z}}}^E_{x,y}(t)-
h\left(\widehat{\vec{z}}^{E}_{x,y}(t),
d\phi\left(\widehat{\vec{z}}^{E}_{x,y}(t)\right)\right)\right] dt\\
\leq \sum_{x,y\in A} \Lambda^E_\lambda(\{x,y\})\int_0^{T_E(x,y)} l\left( \widehat{\vec{z}}^E_{x,y}(t), \dot{\widehat{\vec{z}}}^E_{x,y}(t)\right) dt = \sum_{x,y\in A} \Lambda^E_\lambda(\{x,y\}) C_{T_E(x,y)}(x,y)
\\ =  \sum_{x,y\in A} \Lambda^E_\lambda(\{x,y\})\left[ C_{T_E(x,y)}(x,y) + ET_E(x,y)\right] - E\sum_{x,y\in A} \Lambda^E_\lambda(\{x,y\})T_E(x,y) = \\ \sum_{x,y\in A} \Lambda^E_\lambda(\{x,y\}) D_E(x,y) - E\sum_{x,y\in A} \Lambda^E_\lambda(\{x,y\})T_E(x,y)= \D_E(\lambda)- ED^{'}_E(\lambda)  \ .
\end{multline}
To obtain  (\ref{longineq}) we used the Young inequality  in the second line,  (\ref{de=te})  and (\ref{tineq0}) on the last line.

Since (\ref{longineq}) is valid for any $\phi\in C^1(M)$ we get from this and (\ref{22})  that
 \be\label{last}\D_E(\lambda)- E \D^{'}_E(\lambda) \geq {\cal Q}(\lambda, \mu_\Lambda^E) \geq \oH_T^*(\lambda) = \max_{E\geq \uE} \D_E(\lambda)-TE \ ,  \ee
 holds for {\it any} $E\geq \uE$.
Now, if it so happens that the maximizer $E(\lambda,T)$ on the right of (\ref{last}) is on the complement of the set $N$ in $[\uE,\infty[$, then $D^{'}_E(\lambda)=T=\int_M d\mu^E_\Lambda$ for $E=E(\lambda,T)$ via Lemma~\ref{2side} and the inequality in (\ref{last}) turns into an equality.
  Otherwise, if
$E(\lambda,T)\in N- \{\uE\}$, we take the sequences $E_n\nearrow E(\lambda,T)$, $E^n \searrow E(\lambda,T)$ for $E_n, E^n\in ]\uE, \infty[-N$  and the corresponding limits $\mu^+_\lambda$, $\mu^-_\lambda$ defined in (\ref{mupm}). Since ${\cal Q}_T$ is a convex, l.s.c as a function of $\mu$ we get that the left inequality in (\ref{last}) survives the limit, and
 \be\label{lastt}\D_{E(\lambda,T)}(\lambda)- E(\lambda,T)\frac{d^+}{dE} \D_{E(\lambda,T)}(\lambda) \geq {\cal Q}(\lambda, \mu^+_\lambda) \ \ , \ \  \ \D_{E(\lambda,T)}(\lambda)- {E(\lambda,T)}\frac{d^-}{dE} \D_{E(\lambda,T)}(\lambda) \geq {\cal Q}(\lambda, \mu^-_\lambda) \ , \ee
while $\frac{d^+}{dE} \D_{E(\lambda,T)}(\lambda)= \int d\mu^+_\lambda$ and $\frac{d^-}{dE} \D_{E(\lambda,T)}(\lambda)= \int d\mu^-_\lambda$. Then, upon taking a convex combination $\mu_\lambda=\alpha T^{-1}\mu_\lambda^+ + T^{-1}(1-\alpha)\mu^-_\lambda$ such that, according to Definition~\ref{mulambda},
\be\label{lasttt}\alpha \frac{d^+}{dE} \D_{E(\lambda,T)}(\lambda) + (1-\alpha)\frac{d^-}{dE} \D_{E(\lambda,T)}(\lambda) = T\int d \mu_\lambda =T\ee
and using the convexity of ${\cal Q}$ in $\mu$ we get from (\ref{lastt}, \ref{lasttt})
$$ \D_{E(\lambda,T)}(\lambda)- TE(\lambda,T)\geq {\cal Q}(\lambda, T\mu_\lambda)\equiv {\cal Q}_T(\lambda, \mu_\lambda)$$
This, with the right inequality of (\ref{ntp}) yields the equality ${\cal Q}_T(\lambda, \mu_\lambda)=\oH_T^*(\lambda)$.

Finally, if $E(\lambda,T)=\uE$ we proceed as follows: Let $E^n\searrow \uE$ and $\mu^+_\lambda:=\lim_{n\rightarrow \infty} \mu^{E^n}_\lambda$. It follows that
 \be\label{rrr}\int_M d\mu^+_\lambda = \lim_{n\rightarrow \infty} \int_M d\mu^{E^n}_\lambda = \lim_{n\rightarrow \infty} \D^{'}_{E^n}(\lambda) = \frac{d^+}{dE}\D_{\uE}(\lambda)\in (0, T] \ . \ee
Let $\mu_\lambda$ as in (\ref{part3}).
 From (\ref{calEdef}, , \ref{rrr}) and (\ref{maxham}) we get
\be\label{pth3} {\cal Q}_T(\lambda, \mu_\lambda)\leq {\cal Q}(\lambda, \mu_\lambda^+) + \left(T-\frac{d^+}{dE}\D_{\uE}(\lambda) \right){\cal Q}(0, \mu_M)= {\cal Q}(\lambda,\mu_\lambda^+)-  \left( T- \frac{d^+}{dE}\D_{\uE}(\lambda)\right)
\uE \ee
while (\ref{maxham}) and the left part of (\ref{lastt}) for $E=\uE$ imply
\be\label{lastagain} {\cal Q}(\lambda, \mu_\lambda^+)\leq \D_{\uE}(\lambda) -\uE \frac{d^+}{dE}\D_{\uE}(\lambda) \ . \ee
From
 (\ref{pth3}) and (\ref{lastagain}) we get
 $$ {\cal Q}_T(\lambda,\mu_\lambda)\leq \D_{\uE}(\lambda)-\uE T\leq  \oH^*_T(\lambda) $$
 and the equality holds via (\ref{22}). The last part of Theorem~\ref{th1} follows from the equality in (\ref{22}) as well.
 \end{proof}
\subsection{Proof of Theorem~\ref{th5}}
Theorem~\ref{th1}-(\ref{hcT})  and (\ref{CTmudef}) imply \be\label{hct1}\hC_T(\lambda)=\min_{\mu\in\ooM} \hC_T(\lambda\|\mu) \ . \ee Next, we note that $\D_E(\lambda\|\mu)$ is a concave function of $E$ for $E\geq \uE$. In fact, from  (\ref{HEmore}) and convexity of $h(x,\cdot)$ for each $x\in M$ we obtain
$$\phi_i\in {\cal H}_{E_i} \ \ ,  i=1,2  \ \ \Longrightarrow \alpha\phi_1+ (1-\alpha)\phi_2\in{\cal H}_{\alpha E_1+(1-\alpha)E_2}$$
for $\alpha\in (0,1)$ and $E_1, E_2\geq \uE$. The concavity of $\D_{(\cdot)}(\lambda\|\mu)$ follows from its definition (\ref{Demudef}). Then, by convex duality and (\ref{CTmudef})
$$ \D_E(\lambda\|\mu)=\min_{T>0} \left[\hC_T(\lambda\|\mu) + ET\right] \ . $$
By the same argument
$$ \D_E(\lambda)=\min_{T>0} \left[\hC_T(\lambda) + ET\right] \ . $$
Hence, (\ref{hct1}) and Theorem~\ref{th1}-(\ref{hcT1}) imply
$$ \min_{\mu\in\ooM}\D_E(\lambda\|\mu)= \min_{\mu\in\ooM}\min_{T>0} \left[\hC_T(\lambda\|\mu) + ET\right]$$
$$=
\min_{T>0}\min_{\mu\in\ooM}\left[\hC_T(\lambda\|\mu) + ET\right]= \min_{T>0}\left[\hC_T(\lambda) + ET\right]= \D_E(\lambda) \ .  $$
$\Box$
\section{Proof of Theorems~\ref{th3}\&\ref{th6}}\label{th23}
\subsection{Auxiliary results}
\par\noindent
Lemma~\ref{lsc} follows from the surjectivity of $Exp^{(t)}_l(x)$  as a mapping from $T_xM$ to $M$, for any $x\in M$ and any $t\not= 0$ (Recall definition at Section~\ref{notation}-\ref{Exp}):
\begin{lemma}\label{lsc}
Let $\Lambda\in\oM(M\times M)$. For any $t>0$ there exists a Borel measure $\widehat{\Lambda}^{(t)}\in \oM(TM)$ such that $\left(I\otimes Exp^{(t)}_{(l)}\right)_\# \widehat{\Lambda}^{(t)} = \Lambda$. Here $I\otimes Exp^{(t)}_{(l)}(x,v):= \left(x , Exp^{(t)}_{(l)}(x,v)\right)$.
\end{lemma}
The proof of Lemma~\ref{partof} follows directly from the definition of the optimal plan:
\begin{lemma}\label{partof}
Let $\Lambda$ be a minimizer for (\ref{kantdef}), $B\subset M\times M$ a Borel subset and $\Lambda\lfloor_B$ the restriction of $\Lambda$ to $B$. Let $\mu^0_B$, $\mu^1_B$ the marginals of $\Lambda\lfloor_B$ on the factors  of $M\times M$. Then $\Lambda\lfloor_B$ is an optimal plan for $\ccC\left(\mu^0_B,\mu^1_B\right)$. In addition, if $B_1, B_2\subset M\times M$ are disjoint Borel sets then
$$  \ccC\left(\mu^0_{B_1},\mu^1_{B_1}\right)+\ccC\left(\mu^0_{B_2},\mu^1_{B_2}\right)= \ccC\left(\mu^0_{B_1}+\mu^0_{B_2},\mu^1_{B_1}+\mu^1_{B_2}\right)$$
and $\Lambda\lfloor_{B_1\cup B_2}$ is the optimal plan with respect to  $\ccC\left(\mu^0_{B_1}+\mu^0_{B_2},\mu^1_{B_1}+\mu^1_{B_2}\right)$.
\end{lemma}
Lemma~\ref{partof2} represents the {\it time interpolation} of optimal plans (see \cite{vi1}):
\begin{lemma}\label{partof2}
Given $t>0$ and  $\lambda=\lambda^+-\lambda^-\in\mm$. Let $\Lambda^t\in {\cal P}(\lambda^+,\lambda^-)$ be an optimal plan realizing
$$ \cC_t(\lambda^+,\lambda^-)= \int\int C_t(x,y)\Lambda^t(dxdy)  \ . $$
Let $\widehat{\Lambda}^{(t)}\in \oM(TM)$ given in Lemma~\ref{lsc} for $\Lambda=\Lambda^t$. Let $\lambda_s:= \left(Exp_l^{(s)}\right)_\#\widehat{\Lambda}^{(t)}$. Then, if $0<s<t$,
$$ \cC_s(\lambda^+, \lambda_s) + \cC_{t-s}(\lambda_s, \lambda^-)= \cC_t(\lambda^+, \lambda^-) \ . $$
\end{lemma}
\begin{lemma}\label{lem60}
For any $\lambda^+, \lambda^-\in \ooM$ satisfying $\lambda=\lambda^+-\lambda^-\in\ooM$,
$$ \ccC_T(\lambda^+,\lambda^-)\geq \hC_T(\lambda) \ .  $$
\end{lemma}

\begin{lemma}\label{lem61}
$\hC_T(\lambda\|\mu)$ is l.s.c in the weak-* topology of $\mm\times\ooM$. Assuming  ${\bf H_1}$ and ${\bf H_2}$,
for any
$\lambda\in {\cal M}_0$, $\mu\in\ooM$  there exists a sequence  $\{\tilde{\mu}_n\}=\{\rho_n(x)dx\}\subset \ooM$,
 $\{\tilde{\lambda}_n\}= \{\rho_n(q_n^+-q_n^{-})dx\}\subset {\cal M}_0$ where $\rho_n\in C^\infty(M)$ are positive everywhere, $q_n^\pm\in C^\infty(M)$ non-negatives such that $\tilde{\lambda}_n \rightharpoonup \lambda$, $\tilde{\mu}_n \rightharpoonup \mu$ and
\be\label{lim=} \lim_{n\rightarrow\infty} \hC_T(\tilde{\lambda}_n\|\tilde{\mu}_n) = \hC_T(\lambda\|\mu) \ . \ee
\end{lemma}

\begin{lemma}\label{lem62}
For any $\mu\in\ooM$, $\lambda=\lambda^+-\lambda^-\in\mm$
$$ \liminf_{\eps\rightarrow  0} \eps^{-1}\ccC_{\eps T}(\mu+\eps\lambda^+, \mu+\eps\lambda^-) \geq \hC_T(\lambda\|\mu)  \ . $$
\end{lemma}

\begin{lemma}\label{lem63}
Assume $\mu=\rho(x)dx$ and $\lambda=\rho(q^+-q^-)dx$ where $\rho,q^\pm$ are $C^\infty$ functions, $\rho$ positive everywhere on $M$. Then
$$ \limsup_{\eps\rightarrow 0} \eps^{-1}\ccC_{\eps T}\left(\mu+\eps\lambda^+, \mu+\eps\lambda^-\right) \leq \hC_T(\lambda\|\mu) \ . $$
\end{lemma}
\begin{lemma}\label{lastlema}
For $T>0$,
$$ \hC_T(\lambda)\geq \limsup_{\eps\rightarrow 0}\eps^{-1} \inf_{\mu\in\ooM}\ccC_{\eps T}(\mu+\eps\lambda^+, \mu+\eps\lambda^-) \ . $$
\end{lemma}

\begin{proof} {\it of Lemma~\ref{lem60}}: \ We use the duality representation of the Monge-Kantorovich functional
\cite{vi} to obtain (recall $\lambda^\pm\in \ooM$)
$$ \ccC_T(\lambda^+,\lambda^-) +ET = \sup_{\psi,\phi} \left\{ \int_M \psi d\lambda^--\phi d\lambda^+ \ \ \ , \ \
\phi(y)-\psi(x)\leq C_T(x,y) +ET \right\}$$
By (\ref{dedef1})
$C_T(x,y)+ET\geq D_E(x,y)$ for any $x,y\in M$ so, by (\ref{LE}, \ref{dualdefD})
\begin{multline}\sup_{\psi,\phi} \left\{ \int_M \psi d\lambda^--\phi d\lambda^+ \ , \
\phi(y)-\psi(x)\leq C_T(x,y) +ET \right\}\geq \sup_{\phi} \left\{ \int_M \phi d\lambda \ , \
\phi(y)-\phi(x)\leq D_E(x,y) \right\} \\ = {\cal D}_E(\lambda)\end{multline}
so $$ \ccC_T(\lambda^+, \lambda^-) \geq {\cal D}_E(\lambda)-ET$$
for {\it any} $E\geq \uE$.  By Theorem~\ref{th1}-(\ref{hcT1})
$$ \ccC_T(\lambda^+, \lambda^-) \geq \sup_{E\geq \uE}\cD_E(\lambda)-ET= \hC_T(\lambda) \ . $$
\end{proof}
\begin{proof} {\it of Lemma~\ref{lem61}}: \
From (\ref{Demudef}, \ref{CTmudef}) we obtain
$$ \hC_T(\lambda\|\mu)= \sup_{\phi\in C^1(M)} \int_M \phi d\lambda-Th(x, d\phi) d\mu \ . $$
In particular $\hC_T$ is l.s.c (and convex) on $\mm\times\ooM$.
\par
Let $\eps_n\rightarrow 0$ and $\lambda_n:= \lambda_{\eps_n}:= \delta_{\eps_n} * \lambda\in{\cal M}_0$ defined by
\be\label{ld} \int_M\psi d\lambda_n:= \lambda( \delta_{\eps_n} *\psi) \ \ \ \forall \psi\in C^0(M) \ . \ee
By ${\bf H_1}$,  $\lambda_n \rightharpoonup \lambda$ while $\lambda_n$  are smooth.
First, we observe that $\lim_{n\rightarrow \infty} \lambda_n \rightharpoonup \lambda$. Indeed, for any
$\psi\in C^1(M)$:
$$ \lim_{n\rightarrow \infty} \int_M \psi d\lambda_n =\lim_{n\rightarrow \infty} \lambda\left( \delta_{\eps_n} *\psi\right)
= \lambda(\psi) \ . $$
Next, by Jensen's Theorem and ${\bf H_2}$
  \begin{multline}\label{2d}\int_Mh(x, d\delta_\eps*\phi)d\mu= \int_Mh(x, \delta_\eps*d\phi)d\mu\leq \int_{M\times M}h(x, d\phi(y))\delta_\eps(x,y)d\mu(x)dy \\ \equiv \int_M h(x, d\phi)d\delta_\eps* \mu+ \int_{M\times M}\left[ h(x, d\phi(y))-h(y, d\phi(y)\right]\delta_\eps(x,y)d\mu(x)dy \end{multline}
From  section~\ref{notation}-(\ref{hamdes}) and using $\delta_\eps(x,y)=o(1)$ for $D(x,y)>\delta$,
$$ \int_{M\times M}\left[ h(x, d\phi(y))-h(y, d\phi(y)\right]\delta_\eps(x,y)d\mu(x)dy\leq O(\eps) + o(1) \int_M h(x, d\phi) d\delta_\eps*\mu  \ . $$
Next, define $\mu_n=\delta_{\eps_n}*\mu$. Let $\psi_n$ be the maximizer of $\hC(\lambda_n\|\mu_n)$, that is
$$\hC_T(\lambda_n\|\mu_n)= \int_M\psi_n d\lambda_n-Th(x,d\psi_n) d\mu_n$$
By (\ref{ld}, \ref{2d})
 \begin{multline}\hC_T(\lambda_n\|\mu_n) \leq \int_M \delta_\eps*\psi_n d\lambda- (1-o(1)) \int_M Th(x, d\delta_\eps* \psi_n) d\mu +O(\eps_n)= \\ (1-o(1))\left[ \int_M \delta_\eps*\psi_n \frac{d\lambda}{1-o(1)}-\int_M Th(x, d\delta_\eps* \psi_n) d\mu
\right]+\eps_n \leq (1-o(1))\hC\left( \frac{\lambda}{1-o(1)}\|\mu\right)+\eps_n\end{multline}
We obtained
$$ \limsup_{n\rightarrow\infty}\hC_T(\lambda_n\|\mu_n) \leq \hC_T(\lambda\|\mu)$$
which, together with the l.s.c of $\hC_T$, implies the result.
\end{proof}
\begin{proof} {\it of Lemma~\ref{lem62}}: \
Recall that  the Lax-Oleinik Semigroup acting on $\phi\in C^0(M)$
$$ \psi(x,t)=LO(\phi)_{(t,x)} := \sup_{y\in M}\left[\phi(y)-C_t(x,y)\right]$$
is a viscosity solution of the Hamilton-Jacobi equation $\partial_t\psi-h(x, d\psi)=0$ subjected to $\psi_0=\phi(x)$.
If $\phi\in C^1(M)$ then $\psi$ is a {\it classical solution} on some neighborhood of $t=0$, so
$$  \lim_{T\rightarrow 0}  LO(\phi)_{(T,\cdot)}=\phi  \ \ ; \ \  \lim_{T\rightarrow 0} T^{-1} \left[ LO(\phi)_{(T,x)} -\phi(x)\right] = h(x, d\phi) \ . $$
Then for any $\mu_1, \mu_2\in \ooM$
 \begin{multline}\ccC_T(\mu_1, \mu_2)= \sup_{\phi, \psi\in C^1(M)}\left\{\int_M\phi d\mu_2-\psi d\mu_1 \ \ \ ; \ \ \phi(x)-\psi(y)\leq C_T(x,y) \ \ \forall x,y\in M \right\}= \\ \sup_{\phi\in C^1(M)} \int_M\phi d\mu_2 - LO(\phi)_{(T,x)} d\mu_1
 \end{multline}
 Hence
  \begin{multline}\liminf_{\eps\rightarrow  0} \eps^{-1}\ccC_{\eps T}(\mu+\eps\lambda^+, \mu+\eps\lambda^-) = \\
  \liminf_{\eps\rightarrow  0}  \sup_{\phi\in C^1(M)}\int_M \eps^{-1}\left[ \phi(x)- LO(\phi)_{(\eps T,x)}\right]d\mu + \int_M \phi d\lambda^+ - LO(\phi)_{(\eps T,x)} d\lambda^-  \\
  \geq  \sup_{\phi\in C^1(M)} \lim_{\eps\rightarrow  0} \int_M \eps^{-1}\left[ \phi(x)- LO(\phi)_{(\eps T,x)}\right]d\mu + \int_M \phi d\lambda^+ - LO(\phi)_{(\eps T,x)} d\lambda^- \\
   = \sup_{\phi, \psi\in C^1(M)} \int_M  -Th(x, d\phi) d\mu + \phi d\lambda := \hC_T(\lambda\|\mu)  \ .  \end{multline}
\end{proof}
\begin{proof} {\it of Lemma~\ref{lem63}}: \
We may describe the optimal mapping
 $S_{\eps T}:M\rightarrow M$ associated with $C_{\eps T}(\mu+\eps\lambda^+, \mu+\eps\lambda^-)$ in local coordinates on each chart. It is given by the solution to the Monge-Amp\`{e}re equation
\be\label{comM}det\nabla_x S_{\eps T}= \frac{\rho(x)(1+\eps q^-(x))}{\rho(S_{\eps T}(x)) (1+\eps Tq^+(S_{\eps T}(x))}
\ee
where
\be\label{cpsi} \nabla\psi=-\nabla_x C_{\eps T}(x, S_{\eps T}(x)) \  \ee
and
\be\label{evma}\ccC_{\eps T}\left(\mu+\eps\lambda^+, \mu+\eps\lambda^-\right) = \int_M C_{\eps T}(x, S_{\eps T}(x)) \rho(1+\eps Tq^-)dx\ee

We  recall that the inverse of $\nabla_x C_{\eps T}(x, \cdot)$ with respect to the second variable is $I_d+\eps T\nabla\psi$, to leading order in $\eps$. That is,
\be\label{closeT} \nabla_x C_{\eps T}\left(x, x+\eps T\partial_p h(x, \xi)+(\eps T)^2  Q(x,\xi,\eps)\right)=-\xi \ee
where (here and below) $Q$ is a generic  smooth function of its arguments.

 Hence, $S_{\eps T}$ can be expanded in $\eps$ in terms of $\psi$ as
\be\label{Tdef} S_{\eps T}(x)=x+\eps T h_\xi(x, \nabla\psi) + (\eps T)^2 Q(x, \nabla\psi, \eps) \  \ee

 We now expand the right side of (\ref{comM})  using (\ref{Tdef}) to obtain
  \be\label{second}1+ \eps T\left[ q^-(x)-q^+(x) - h_\xi(x, d\psi) \cdot \nabla_x\ln\rho(x)\right]+(\eps T)^2 Q(x, \nabla\psi, x,\eps)\ee
  while the left hand side is
  \be\label{third} det(\nabla_xS_{\eps T})=1+\eps T\nabla\cdot h_\xi(x, d\psi)+(\eps T)^2Q(x, \nabla\psi, \nabla\nabla\psi, x,\eps)\ee
  Comparing  (\ref{second}, \ref{third}),  divide by $\eps T$ and multiply by $\rho$ to obtain
  \be\label{fulleq}T\nabla\cdot \left(\rho h_\xi(x, d\psi)\right)= \rho(q^--q^+) + \eps T\rho Q(x, \nabla\psi, \nabla\nabla \psi, x,\eps) \ .
 \ee

 Now, we substitute $\eps=0$ and get a quasi-linear equation for $\psi_0$:
 \be\label{eq0}T\nabla\cdot \left(\rho h_\xi(x, d\psi_0)\right)= \rho(q^--q^+)  \ .
 \ee
 $\psi_0$ is a maximizer of
 $$ \hC_T(\lambda\|\mu)=\int_M \rho(q^+-q^-)\psi_0 -\int_M \rho Th(x, d\psi_0) dx$$
  By elliptic regularity, $\psi_0\in C^\infty(M)$. Multiply (\ref{eq0}) by $\psi_0$ and integrate over $M$ to obtain
  $$ \int_M\rho(q^+-q^-) = \int_M \rho Th_\xi(x, d\psi_0)\cdot \nabla \psi_0$$
   Then by the Lagrangian/Hamiltonian duality
    \be\label{go1}\hC_T(\lambda\|\mu)= \int_M \rho T\left[ \nabla\psi_0\cdot h_\xi(x, d\psi_0)-h(x, d\psi_0)\right] \equiv T\int_M \rho l\left(x, h_\xi(x, d\psi_0)\right) \ . \ee
   We observe $l\left( x, \frac{y-x}{T}\right) \geq T^{-1}C_T(x,y)$. So, (\ref{evma}) with (\ref{Tdef}) imply
    \be\label{go2}(\eps T )^{-1}\ccC_{\eps T}\left(\mu+\eps\lambda^+, \mu+\eps\lambda^-\right) \leq \int_M \rho(1+\eps Tq^-) l\left(x, h_\xi(x, \nabla\psi_\eps+ \eps T Q(x, \nabla\psi_\eps,\eps )\right)\ee
   where $\psi_\eps$ is a solution of (\ref{fulleq}). Now, if we show that  $\lim_{\eps\rightarrow 0} \psi_\eps=\psi_0$
   in $C^1(M)$ then,  from (\ref{go1}, \ref{go2})
   $$ \limsup_{\eps\rightarrow 0}(\eps)^{-1}\ccC_{\eps T}\left(\mu+\eps\lambda^+, \mu+\eps\lambda^-\right) \leq  T\int_M \rho l\left(x, h_\xi(x, d\psi_0)\right)= \hC(\lambda\|\mu) \ .  $$
   Next we show that, indeed, $\lim_{\eps\rightarrow 0} \psi_\eps=\psi_0$
   in $C^1(M)$.
   \par
  Substitute $\psi_\eps=\psi_0+\phi_\eps$ in (\ref{fulleq}). We obtain
    \be\label{firstl}\nabla\cdot( \sigma (x) \nabla\phi_\eps)= \eps Q(x, \nabla\phi_\eps, \nabla\nabla\phi_\eps,\eps) + \nabla\cdot\left( \rho\langle \nabla^t\phi_\eps,  \tilde{Q}(x, \nabla\phi, \eps)\cdot\nabla\phi_\eps\rangle\right)\ee
   where $\sigma:=Th_{\xi\xi}(x, \nabla\psi_0(x))$ is a positive definite form, while  $\tilde{Q}$ is a smooth matrix valued functions in both $x$ and $\eps$, determined by $\nabla\psi_0$ and $Q$ as given in (\ref{fulleq}).  A direct application of the implicit function theorem implies the existence of a  branch $(\lambda(\eps), \eta_\eps)$ of solutions for
      \be\label{secondl}\nabla\cdot( \sigma (x) \nabla\eta)= \eps Q(x, \nabla\eta, \nabla\nabla\eta,\eps) + \nabla\cdot\left( \rho\langle \nabla^t\eta,  \tilde{Q}(x, \nabla\eta, \eps)\circ\nabla\eta\rangle\right)+ \lambda(\eps)\ee
     where $\eta_0=\lambda(0)=0$ and $\eps\mapsto \eta_\eps$ is (at least)  continuous in $C^1(M)\perp 1$. Note that for $\eps\not=0$ we may have a non-zero $\lambda(\eps)$ which follows from projecting the
     right side on the equation to the Hilbert space perpendicular to constants (recall that $M$ is a compact manifold without boundary, and the left side is surjective on this space). We now show that $\eta_\eps=\phi_\eps$, i.e $\lambda(\eps)=0$ also for  $\eps\not=0$. Indeed, (\ref{firstl}) is equivalent to (\ref{comM}) multiplied by $\rho(x)/\eps$, so (\ref{secondl}) is equivalent to
     $$ det\nabla_x \hat{S}_{\eps T}= \frac{\rho(x)(1+\eps q^-(x))}{\rho(\hat{S}_{\eps T}(x)) (1+\eps q^+(\hat{S}_{\eps T}(x))}+ \eps\rho^{-1}(x)\lambda(\eps)$$
     where
          $\hat{S}_{\eps T}(x)$ obtained from (\ref{Tdef}) with $\psi_\eps:= \psi_0+\eta_\eps$.

     Hence
      \begin{multline}\label{ml}\int_M \left( \rho(\hat{S}_{\eps T}(x)) (1+\eps q^+(\hat{S}_{\eps T}(x))\right)det(\nabla_x\hat{S}_{\eps T}) = \int_M \left( \rho(x)(1+\eps q^-(x))\right) \\ +  \eps\lambda(\eps) \int_M \frac{\rho(\hat{S}_{\eps T}(x)) }{\rho(x)} (1+\eps q^+(\hat{S}_{\eps T}(x))\end{multline}
      However, $\hat{S}_{\eps T}(x)=x+O(\eps)$ is a diffeomorphism on $M$, so
      \begin{multline}\int_M \left( \rho(\hat{S}_{\eps T}(x)) (1+\eps q^+(\hat{S}_{\eps T}(x))\right)det(\nabla_x\hat{S}_{\eps T})=\int_M \left( \rho(\hat{S}_{\eps T}(x)) (1+Tq^+(\hat{S}_{\eps T}(x))\right)|det(\nabla_x\hat{S}_{\eps T})| \\ =
      \int_M  \rho(x) (1+\eps q^+(x))\equiv  \int_M  \rho(x) (1+\eps q^-(x)) \ .
      \end{multline}
      It follows that
      $$ \eps\lambda(\eps) \int_M \frac{ \rho(\hat{S}_{\eps T}(x)) }{\rho(x)}(1+\eps q^+(\hat{S}_{\eps T}(x))=0 \ . $$
      Since $\rho$ is positive everywhere  it follows that
       $\lambda(\eps)\equiv 0$ for $|\eps|$ sufficiently small. We proved that $\eta_\eps\equiv \phi_\eps$ and, in particular, $\phi_\eps\rightarrow 0$ as $\eps\rightarrow 0$ in $C^1\perp 1$, which implies the convergence of $\psi_\eps$ to $\psi_0$ at $\eps\rightarrow 0$ in $C^1\perp 1$.
 \end{proof}

\begin{proof}(of Lemma~\ref{lastlema})
 Given $\eps >0$ let
 \be \label{DEn} D_E^{\eps }(x,y):= \inf_{n\in \mathbb{N}}\left[ C_{\eps n T}(x,y)+\eps n E T\right] \ . \ee
 Evidently, $D^\eps_E(x,y)$ is continuous on $M\times M$ locally uniformly in $E\geq \uE$. Moreover,
 \be\label{limDE}\lim_{\eps \searrow 0}D^\eps_E= D_E\ee
 uniformly on $M\times M$ and locally uniformly in $E\geq \uE$ as well.
 \par
 We now decompose $M\times M$ into mutually disjoint Borel sets
$Q_{n}$:
$$ M\times M = \cup_n Q_{n}^{\eps}  \ , \ \ Q_{n}^{\eps}\cap Q_{E,n^{'}}^{\eps} =\emptyset \ \ \text{if} \ n\not= n^{'}$$ such that
 $$ Q_{n}^{\eps}\subset
 \left\{ (x,y)\in M\times M \ ;  \ \ D_E^{\eps }(x,y)= C_{\eps n T}(x,y)+\eps nE T\right\}  \ . $$

 Let $\Lambda^E_\eps\in {\cal P}(\lambda^+,\lambda^-)$ be an optimal plan for
  \be\label{calDen}{\cal D}^{\eps }_E(\lambda)= \int_{M\times M} D_E^{\eps }(x,y) d\Lambda^E_\eps= \min_{\Lambda\in {\cal P}(\lambda^+,\lambda^-)}\int_{M\times M} D_E^{\eps }(x,y) d\Lambda \ , \ee
 and $\Lambda_\eps^{n}= \Lambda^E_\eps\lfloor_{Q_{n}^\eps}$,  the restriction of $\Lambda^E_\eps$ to $Q^\eps_{n}$. Set  $\lambda^\pm_n$ to be the marginals of $\Lambda_\eps^{n}$
 on the first and second factors of $M\times M$. Then
 $\sum_{n=1}^\infty \Lambda_\eps^{n} = \Lambda^E_\eps$ and \be\label{sumlambda}  \sum_{n=1}^\infty \lambda^\pm_n = \lambda^\pm \ee
 \begin{remark}
 Note that $Q^\eps_{n}=\emptyset$ for all but a finite number of $n\in \mathbb{N}$. In particular, the sum
 (\ref{sumlambda}) contains only a finite number of non-zero terms.
 \end{remark}
 Let $|\lambda_n|:= \int_M d\lambda^\pm_n\equiv \int_{M\times M} d\Lambda_n^\eps$. The {\it averaged flight time} is
 \be\label{Tav}\langle T\rangle^\eps:= \eps T\sum_{n=1}^\infty n |\lambda_n| \ee
 We observe that $\langle T\rangle^\eps\in \partial_E{\cal D}_E^{\eps}(\lambda)$, where $\partial_E$ is the super gradient as a function of $E$.  At this stage we choose $E$ depending on $\eps, T$ such that
 \be\label{Teav} \langle T\rangle^\eps= T + 2\eps T |\lambda^\pm|\ee
We now apply Lemma~\ref{lsc}:
 Recalling Section~\ref{notation}-\ref{Exp}, let $\widehat{\Lambda}_\eps^{n}\in \oM(TM)$ satisfying \\  $\left(I\oplus Exp_{(l)}^{(t=\eps nT)}\right)_\# \widehat{\Lambda}_\eps^{n} = \Lambda_\eps^{n}$.
Use $\widehat{\Lambda}_\eps^{n}$ to define
 $\lambda_n^j:= \left(Exp_{(l)}^{(t=\eps nT)}\right)_\# \widehat{\Lambda}_\eps^{n}\in \oM(M)$ for $j=0,1\ldots n$. Note that
 \be\label{ec} \lambda_n^0=\lambda_n^+ \ \ \ ,   \lambda_n^n=\lambda_n^- \ . \ee
 By Lemma~\ref{partof2}
 \be\label{sumcn} \ccC_{\eps n T} (\lambda_n^+, \lambda_n^-)+\eps nE T|\lambda_n|= \sum_{j=0}^{n-1} \left[\ccC_{\eps T}(\lambda_n^j, \lambda^{j+1}_n)+ \eps ET|\lambda_n|\right]\ee
From (\ref{DEn}, \ref{calDen},  \ref{sumlambda}, \ref{sumcn}) and Lemma~\ref{partof}
\be\label{cnineq}{\cal D}^\eps_E(\lambda)= \sum_{n=1}^\infty \cD_E^\eps(\lambda_n)=   \sum_{n=1}^\infty \left[\ccC_{\eps nT}(\lambda_n^+, \lambda_n^-)+\eps nE T|\lambda_n|\right]=\sum_{n=1}^\infty \sum_{j=0}^{n-1}\left(\ccC_{\eps T}(\lambda_n^j, \lambda_n^{j+1})+ \eps E T|\lambda_n| \right) \ . \ee
Let now
$$ \mu^{\eps, E}=\eps\sum_{n=1}^\infty \sum_{j=1}^{n-1} \lambda_n^j \ \ . $$
Note that
$$ \mu^{\eps, E}=\eps\sum_{n=1}^\infty \sum_{j=0}^{n} \lambda_n^j-\eps\sum_{n=1}^\infty  \lambda_n^0
-\eps\sum_{n=1}^\infty  \lambda_n^n  \ . $$
By (\ref{sumlambda},\ref{ec}, \ref{Tav}) we obtain
\be\label{oom1} \left|\mu^{\eps, E}  \right| = \eps\sum_{n=1}^\infty (n+1)|\lambda_n^\pm|-2\eps |\lambda^\pm| = 1 \ \Longrightarrow
\mu^{\eps, E}\in \ooM  \ . \ee
By (\ref{sumlambda}, \ref{ec})
\begin{multline}\label{ttt}\sum_{n=1}^\infty \sum_{j=0}^{n-1}\ccC_{\eps T}(\lambda_n^j, \lambda_n^{j+1}) \geq \ccC_{\eps T}\left( \sum_{n=1}^\infty \sum_{j=0}^{n-1} \lambda_n^j, \sum_{n=1}^\infty \sum_{j=1}^{n} \lambda_n^{j+1}\right)= \eps^{-1}  \ccC_{\eps T}\left( \eps\sum_{n=1}^\infty \sum_{j=0}^{n-1} \lambda_n^j, \eps\sum_{n=1}^\infty \sum_{j=1}^{n} \lambda_n^{j+1}\right) \\ =\eps^{-1}  \ccC_{\eps T}\left(\mu^{\eps, E}+\eps\lambda^+, \mu^{\eps, E}+\eps\lambda^-\right) \ .
\end{multline}
From (\ref{Tav}, \ref{cnineq}, \ref{ttt} , \ref{oom1})
\be\label{sdter}{\cal D}^\eps_E(\lambda)- \langle T\rangle^\eps E \geq \eps^{-1}  \ccC_{\eps T}\left(\mu^{\eps, E}+\eps\lambda^+, \mu^{\eps, E}+\eps\lambda^-\right)\geq \eps^{-1}  \inf_{\mu\in \ooM}\ccC_{\eps T}\left(\mu+\eps\lambda^+, \mu+\eps\lambda^-\right) \ . \ee
Finally,  Theorem~\ref{th1}-3,  (\ref{limDE}, \ref{Teav}, \ref{sdter}) imply
$$ \hC_T(\lambda)\geq  {\cal D}_E(\lambda)- T E=  \lim_{\eps\rightarrow 0} {\cal D}^\eps_E(\lambda)- \langle T\rangle^\eps E\geq \limsup_{\eps\rightarrow 0}\eps^{-1}  \inf_{\mu\in \ooM}\ccC_{\eps T}\left(\mu+\eps\lambda^+, \mu+\eps\lambda^-\right) \ . $$
\end{proof}
\subsection{Proof of theorem~\ref{th3}}
From Theorem~\ref{th1}- (1) we get
$$\hC_{\eps T}(\eps\lambda) = \eps \hC_T(\lambda) \ . $$
We now apply  Lemma~\ref{lem60}, adapted to the case where $|\lambda^\pm|:= \int\lambda^\pm \not=1$. Then
$$ \ccC_T(\lambda^+,\lambda^-)=  |\lambda^\pm|\ccC_T\left(\frac{\lambda^+}{|\lambda^+|},\frac{\lambda^-}{|\lambda^-|}\right)\geq |\lambda^\pm|\hC_T\left(\frac{\lambda}{|\lambda^\pm|}\right)=\hC_{T/|\lambda^\pm|}\left(\lambda\right)  \ .  $$

Note that $\int d\mu+\eps d\lambda^\pm = 1+O(\eps)$, hence
$$ \eps^{-1}\ccC_{\eps T}\left( \mu+\eps\lambda^+, \mu+\eps\lambda^-\right) \geq  \hC_{T_\eps}(\lambda) $$
where $T_\eps\rightarrow T$ as $\eps\rightarrow 0$.
Hence
$$ \liminf_{\eps\rightarrow 0} \inf_{\ooM}\eps^{-1}\ccC_{\eps T}\left( \mu+\eps\lambda^+, \mu+\eps\lambda^-\right)
\geq \hC_T(\lambda) \ . $$
The Theorem follows from this and Lemma~\ref{lastlema}.

$\Box$
\subsection{Proof of Theorem~\ref{th6}}

 We have to show that for any $(\mu, \lambda)\in \ooM\times \mm$ and any sequence $(\mu_n, \lambda_n) \rightharpoonup (\mu, \lambda)$ as $n\rightarrow \infty$:
 \be\label{lem62too}\liminf_{n\rightarrow \infty} n \ccC_{T/n}\left( \mu_n+n^{-1}\lambda_n^+, \mu_n+n^{-1}\lambda_n^-\right)\geq \hC(\lambda\|\mu)\ee
and, in addition, {\it there exists}  a sequence $(\hat{\mu}_n, \hat{\lambda}_n) \rightharpoonup (\mu, \lambda)$ for which
 \be\label{gamze}\lim_{n\rightarrow \infty} n\ccC_{ T/n}\left( \hat{\mu}_n+n^{-1}\hat{\lambda}_n^+, \hat{\mu}_n+n^{-1}\hat{\lambda}_n^-\right)= \hC(\lambda\|\mu) \ . \ee
The inequality (\ref{lem62too}) follows directly from Lemma~\ref{lem62}. To prove (\ref{gamze}), we first consider the sequence $(\tilde{\mu}_n, \tilde{\lambda}_n)$ subjected to Lemma~\ref{lem61}.
From Lemma~\ref{lem63} and Lemma~\ref{lem61},
$$ \lim_{j\rightarrow\infty} \limsup_{n\rightarrow\infty}  \ n \ccC_{T/n} \left( \tilde{\mu}_j+ n^{-1}\tilde{\lambda}^+_j, \tilde{\mu}_j+ n^{-1}\tilde{\lambda}^-_j\right) \leq \lim_{j\rightarrow\infty} \hC_T\left(\tilde{\lambda}_j\|\tilde{\mu}_j\right)= \hC(\lambda\|\mu) \ . $$
So, there exists a  subsequence $j_n$ along  which
$$  \limsup_{n\rightarrow\infty}  \ n\ccC_{T/n} \left( \tilde{\mu}_{j_n}+ n^{-1}\tilde{\lambda}^+_{j_n}, \tilde{\mu}_{j_n}+ n^{-1}\tilde{\lambda}^-_{j_n}\right) \leq\hC(\lambda\|\mu) \ . $$
This, with (\ref{lem62too}), implies (\ref{gamze}).
\par
The second part of the theorem follows from (\ref{lem62too}) and Theorem~\ref{th3}.

$\Box$

\end{document}